\documentclass[12pt]{elsarticle}
\usepackage{mystyleElsArticle}
\usepackage{mathrsfs}
\usepackage{booktabs}

\usepackage[margin=0.93 in]{geometry}    % Margin: 1-inch all round

\renewcommand\d{\mathrm{d}}

\newcommand{\veps}{\varepsilon}
\def\epsilon{\varepsilon}
\def\bx{\textbf{x}}
\def\bc{\textbf{c}}
\newcommand{\secref}[1]{\S\ref{#1}}

\newtheorem{example}[theorem]{Example}

\makeatletter
\def\ps@pprintTitle{%
	\let\@oddhead\@empty
	\let\@evenhead\@empty
	\let\@oddfoot\@empty
	\let\@evenfoot\@oddfoot
}
\makeatother

\begin{document}
	
	\begin{frontmatter}
		\title{A multiscale reduced basis method for Schr\"{o}dinger equation with multiscale and random potentials}
		\author[SoochowUniv]{Jingrun Chen}
		\ead{jingrunchen@suda.edu.cn}
		\author[hku]{Dingjiong Ma}
		\ead{martin35@hku.hk}
		\author[hku]{Zhiwen Zhang\corref{cor1}}
		\ead{zhangzw@hku.hk}
		
		\address[SoochowUniv]{Mathematical Center for Interdisciplinary Research and School of Mathematical Sciences, Soochow University, Suzhou, China.}
		\address[hku]{Department of Mathematics, The University of Hong Kong, Pokfulam Road, Hong Kong SAR, China.}
		\cortext[cor1]{Corresponding author}		
		
\begin{abstract}
\noindent	
The semiclassical Schr\"{o}dinger equation with multiscale and random potentials often appears when studying electron dynamics in heterogeneous quantum systems.
As time evolves, the wavefunction develops high-frequency oscillations in both the physical space and the random space, which poses severe challenges
for numerical methods. In this paper, we propose a multiscale reduced basis method, where we construct multiscale reduced basis functions using an optimization method and the proper orthogonal decomposition method in the physical space and employ the quasi-Monte Carlo method in the random space. 
Our method is verified to be efficient: the spatial gridsize is only proportional to the semiclassical parameter and the number of samples 
in the random space is inversely proportional to the same parameter. Several theoretical aspects of the proposed method, including how to determine the number of samples in the construction of multiscale reduced basis and convergence analysis, are studied with numerical justification. In addition, we investigate the Anderson localization phenomena for Schr\"{o}dinger equation with correlated random potentials in both 1D and 2D.
% 65M60  	Finite elements, Rayleigh-Ritz and Galerkin methods, finite methods
% 74Q10  	Homogenization and oscillations in dynamical problems
% 35J10  	Schrödinger operator 
% 35Q41  	Time-dependent Schrödinger equations, Dirac equations
% 65K10  	Optimization and variational techniques			
\end{abstract}
\end{frontmatter}
	% \newpage
	% \tableofcontents
	% \listoffigures
	% \listoftables
	% \newpage
\section{Introduction} \label{sec:introduction}
\noindent
The semiclassical Schr\"{o}dinger equation describes electron dynamics in the semiclassical regime. Applications of such an equation
can be found in Bose-Einstein condensation, graphene, semiconductors, topological insulators, etc. When propagating in a (quasi-)periodic microstructure,
electrons experience a multiscale potential. As a consequence, the electron wavefunction develops high-frequency oscillations, which poses severe challenges
from the numerical perspective. Brute-force methods are very costly and asymptotics-based methods have been proposed in the literature; see \cite{jin2011mathematical}
for review and references therein.

In \cite{Anderson:58}, Anderson proposed to study localized eigenstates in a tight-binding model with random potentials. This model was soon to be generalized
to the random Schr\"{o}dinger equation, i.e., the Schr\"{o}dinger equation with a random potential. In this case, electrons are found to be localized 
provided that the strength of randomness is sufficiently large. The randomness can be realized in an experiment by enhancing the disorder of impurities
in a material. Due to the importance of this model, Anderson was awarded the Nobel Prize in physics in 1977. In the presence of multiscale and random
potentials, the electron wavefunction develops high-frequency oscillations in both the physical space and the random space, making numerical approximations
even more difficult.

In this paper, we study the following Schr\"{o}dinger equation with random potential in the semiclassical regime
\begin{equation}
\left\{
\begin{aligned}
i\epsilon\partial_t\psi^\epsilon&=-\frac{\epsilon^2}{2}\Delta\psi^\epsilon+v^{\epsilon}(\bx,\omega) \psi^\epsilon, \quad \bx\in D,  \quad \omega\in\Omega, \quad t\in \mathbb{R},\\
\psi^\epsilon &\in H_{\textrm{P}}^{1}(D), \quad \omega\in\Omega, \quad t\in \mathbb{R},\\
\psi^\epsilon|_{t=0}&=\psi_{\textrm{in}}(\bx),\quad \bx\in D,
\end{aligned}
\right.
\label{eqn:Sch}
\end{equation}
where $ 0<\epsilon \ll1 $ is an effective Planck constant describing the microscopic and macroscopic scale ratio, $ d $ is the spatial dimension, $ v^{\epsilon}(\bx,\omega) $ is the given random potential, $\psi^\epsilon = \psi^\epsilon(t,\bx,\omega) $ is the electron wavefunction, and $ \psi_{\textrm{in}}(\bx) $ is the initial data.  Here $ D = [0,1]^d $ is the spatial domain and $ H_{\textrm{P}}^{1}(D) = \{\psi |\psi\in H^1(D) \textrm{ and } \psi \textrm{ is periodic over D} \}$. 

Equation \eqref{eqn:Sch} can be used to model electron transport in a disordered medium in a single-electron picture where the electron interaction is ignored. It is customary to write the semiclassical Schr\"{o}dinger equation and the multiscale and random potential with a single parameter $\epsilon$. But there is no reason that the parameter of the multiscale and random potential should be the same as the semiclassical parameter; see \secref{sec:NumericalExamples} for details on the parameterization of the multiscale and random potential $v^{\epsilon}(\bx,\omega)$.

%Note that \eqref{eqn:Sch} is only defined over $D$ for simplicity though the original Schr\"{o}dinger equation is defined in $\mathbb{R}^d$ and artificial boundary conditions shall be used here. 

The existence of Anderson localization is closely related to the electron wavefunction in \eqref{eqn:Sch}. To be specific, assume $\psi^\epsilon(t,\bx,\omega)$ has zero mean with respect to the measure $\rho$ induced by $v^{\epsilon}(\bx,\omega)$ and denote $A(t)=\mathbb{E}[\int_{R^d}|\bx|^2|\psi^{\epsilon}|^2 d\bx]_{\rho}$ the second-order moment of the position density. When the strength of disorder is small, an electron undergoes a diffusion process with $A(t)=2C_{d}t$, $C_{d}>0$. In the presence of a strong disorder, however, $A(t)$ converges to a time-independent quantity, i.e., $\lim_{t\rightarrow \infty}A(t)=C$, which implies the localization of the electron and the system undergoes a metal-insulator transition \cite{mott1990metal,Erdos:10}. When $d=1$, localization always occurs for \eqref{eqn:Sch} with random potential \cite{Anderson:58}. When $d\geq 2$, the situation becomes complicated. Some analytical results show that localization occurs when the strength of disorder is large \cite{frohlich1983absence,aizenman1993localization}. This motivates us to study Anderson localization
in the presence of correlated random potentials \cite{nosov2019correlation}.

When the potential is deterministic, i.e., $v^{\epsilon}(\bx,\omega) = v^{\epsilon}(\bx)$, many numerical methods have been proposed; see \cite{bao2002time, faou2006poisson, tanushev2007mountain, jin2008gaussian, faou2009computing, delgadillo2016gauge, chen2019multiscale} for example. 
When the potential is random, few works have been done; see \cite{wu2016bloch,jin2019gaussian}.
As mentioned above, the major difficulty is that the wavefunction $\psi^\epsilon$ develops high-frequency oscillations in both the physical space and the random space, which requires tremendous computational resources.

Our work is motivated by the multiscale finite element method (FEM) for solving elliptic problems with multiscale coefficients \cite{hou1999convergence,efendiev2009multiscale}. The multiscale FEM is capable of correctly capturing the large scale components of the multiscale solution on a coarse grid without accurately resolving all the small scale features in the solution. This is accomplished by incorporating the local microstructures of the differential operator into the multiscale FEM basis functions. Recently, several relevant works on constructing localized basis functions that approximate the elliptic operator with heterogeneous coefficients have been proposed. In \cite{maalqvist2014localization}, Malqvist and Peterseim construct localized multiscale basis functions using a modified variational multiscale method. The exponentially decaying property of these modified basis has been shown both theoretically and numerically. Meanwhile, Owhadi \cite{Owhadi:2015,owhadi2017multigrid} reformulates the multiscale problem from the perspective of decision theory using the idea of gamblets as the modified basis. Hou et.al. \cite{hou2017sparse} extend these works such that localized basis functions can also be constructed for higher-order strongly elliptic operators.  Recently, Hou, Ma, and Zhang propose to build localized multiscale stochastic basis to solve elliptic problems with multiscale and random coefficients \cite{hou2018model}.

In this paper, we propose a multiscale reduced basis method to solve the Schr\"{o}dinger equation with random potentials in the semiclassical regime. Our method consists of offline and online stages. In the offline stage, we apply an optimization approach to systematically construct localized multiscale reduced basis functions on each patch associated with each coarse gridpoint. These basis functions provide nearly optimal approximation to the random Schr\"{o}dinger operator. 
In the online stage, we use these basis functions to approximate the physical space of the solution and the quasi-Monte Carlo (qMC) method to approximate the random space of the solution, respectively. We find the proposed method is efficient in the sense that
the number of basis functions is only proportional to $\veps$ and the number of samples in qMC is inversely proportional to $\veps$.
Under some conditions, we conduct the convergence analysis of the proposed method with numerical verifications. 
Moreover, we study how to determine the number of samples in qMC such that the corresponding multiscale reduced basis functions provide accurate approximation
of the solution space. Finally we investigate the existence of Anderson localization for correlated random potentials.

The rest of the paper is organized as follows. For completeness, in \secref{sec:DeterministicCase}, we introduce multiscale basis functions for the deterministic Schr\"{o}dinger equation in semiclassical regime and discuss some properties of the basis functions. In \secref{sec:RandomCase}, we propose a multiscale reduced basis method to solve the random Schr\"{o}dinger equation. Analysis results are presented in \secref{sec:analysis} and numerical experiments, including both 1D and 2D examples, are conducted to demonstrate the convergence and efficiency of the proposed method in \secref{sec:NumericalExamples}. Conclusions and discussions are drawn in \secref{sec:Conclusion}.
  
\section{Multiscale basis functions for deterministic Schr\"{o}dinger equations} \label{sec:DeterministicCase}
\noindent
In this section, we briefly review the construction of multiscale basis functions based on an optimization approach to solve the Schr\"{o}dinger equation with a deterministic potential. Some properties of the multiscale basis functions are also given.
	
\subsection{Construction of multiscale basis functions} \label{sec:OC}
\noindent
In the deterministic case, we consider the following problem
\begin{equation}
\left\{
\begin{aligned}
i\epsilon\partial_t\psi^\epsilon&=-\frac{\epsilon^2}{2}\Delta\psi^\epsilon+v^{\epsilon}(\bx) \psi^\epsilon,\quad \bx\in D,\quad t\in \mathbb{R},\\
\psi^\epsilon &\in H_{\textrm{P}}^{1}(D),\\
\psi^\epsilon|_{t=0}&=\psi_{\textrm{in}}(\bx).
\end{aligned}
\right.
\label{Sch}
\end{equation}
$ \psi_{\textrm{in}}(\bx) $ is the initial data over $D$. Defining the Hamiltonian operator $\mathcal{H}(\cdot)\equiv -\frac{\epsilon^2}{2}\Delta(\cdot)+v^{\epsilon}(\bx)(\cdot) $ and introducing the following energy notation $ ||\cdot||_V $ for Hamiltonian operator
\begin{align}\label{eqn:energynorm}
||\psi^\epsilon||_V=\frac12(\mathcal{H}\psi^\epsilon,\psi^\epsilon)=\frac12\int_{D}\frac{\epsilon^2}{2}|\nabla \psi^\epsilon|^2+v^{\epsilon}(\bx)|\psi^\epsilon|^2 \d\bx.
\end{align}
Note that \eqref{eqn:energynorm} does not define a norm since $v^{\epsilon}$ usually can be negative, and thus the bilinear form associated to this notation is not coercive, which is quite different from the case of elliptic equations. However, this does not mean that available approaches \cite{hou1997multiscale, babuska2011optimal, maalqvist2014localization, owhadi2017multigrid, hou2017sparse} cannot be used for the Schr\"{o}dinger equation. In fact, we shall utilize the similar idea to construct localized multiscale basis functions on a coarse mesh by an optimization approach using the above energy notation $ ||\cdot||_V $ for the Hamiltonian operator.

%\textcolor{red}{??? Fix the notations, $\varphi^{H}_{i}(\bx)$, $\varphi^{h}_{i}(\bx)$ and $\phi_i(\bx)$.}	
To construct such localized multiscale basis functions,  we first partition the physical domain $D$ into a set of regular coarse elements with mesh size $H$. For example, we divide $D$ into a set of non-overlapping triangles $\mathcal{T}_{H}=\cup\{K\}$, such that no vertex of one triangle lies in the interior of the edge of another triangle. On each element $K$, we define a set of nodal basis $\{\varphi_{j,K},j=1,...,k\}$ with $k$ being the number of nodes of the element. From now on, we neglect the subscript $K$ for notational convenience. The functions $\varphi_{i}(\bx)$ are called measurement functions, which are chosen as the characteristic functions on each coarse element in \cite{hou2017sparse,owhadi2017multigrid} and piecewise linear basis functions in \cite{maalqvist2014localization}. In \cite{li2017computing,hou2018model}, it is found that the usage of FEM nodal basis functions reduces the approximation error and thus the same setting is adopted in the current work.
	
Let $\mathcal{N}$ denote the set of vertices of $\mathcal{T}_{H}$ (removing the repeated vertices due to the periodic boundary condition) and $N_{H}$ be the number of vertices. For every vertex $\bx_i\in\mathcal{N}$, let $\varphi^{H}_{i}(\bx)$ denote the corresponding nodal basis function, i.e., $\varphi^{H}_{i}(\bx_j)=\delta_{ij}$. Since all the nodal basis functions $\varphi_{i}(\bx)$ are continuous across the boundaries of the elements, we have
\begin{align*}
V^{H}=\{\varphi^{H}_{i}(\bx):i=1,...,N_H \}\subset H_{\textrm{P}}^{1}(D).
%	\label{coarse_space}
\end{align*}
Then, we can solve optimization problems to obtain the multiscale basis functions. Specifically, let $\phi_i(\bx)$ be the minimizer of the following constrained optimization problem
\begin{align}
\phi_i& =\underset{\phi \in H_{\textrm{P}}^{1}(D)}{\arg\min} ||\phi||_V   \label{OC_SchGLBBasis_Obj}\\
\text{s.t.}\ &\int_{D}\phi \varphi^{H}_{j} \d\bx= \delta_{i,j},\ \forall 1\leq j \leq N_{H}.  \label{OC_SchGLBBasis_Cons1}
\end{align}
The superscript $\epsilon$ is dropped for notation simplicity and the periodic boundary condition is incorporated into the above optimization problem through the solution space $H_{\textrm{P}}^{1}(D)$. 
	
In general, one cannot solve the above optimization problem analytically. Therefore, we use numerical methods to solve it. Specifically, we partition the physical domain $D$ into a set of non-overlapping fine triangles with size $h \ll \veps$. Then, we use standard FEM to discretize $\phi_i(\bx)$, $\varphi^{H}_{j}(\bx)$, $1\leq i,j\leq N_{H}$. In the discrete level, the optimization problem \eqref{OC_SchGLBBasis_Obj}-\eqref{OC_SchGLBBasis_Cons1} is reduced to a constrained quadratic optimization problem; see \eqref{eqn:QP} in Section \ref{sec:DetermineSamples}, which can be efficiently solved using Lagrange multiplier methods. Finally, with these multiscale FEM basis functions $ \{\phi_i(\bx)\}_{i=1}^{N_H} $, we can solve the Schr\"{o}dinger equation \eqref{Sch} using the Galerkin method.
\begin{remark}
In analogy to the multistate FEM \cite{hou1999convergence,efendiev2009multiscale},  the multiscale basis functions $\{\phi_i(\bx)\}_{i=1}^{N_H}$ are defined on coarse elements with mesh size $H$. However, they are represented by fine-scale FEM basis with mesh size $h$, which can be pre-computed and done in parallel. 
\end{remark}	
\begin{remark}
The notation $||\cdot||_V$ in \eqref{eqn:energynorm} does not define a norm. However, as long as the potential $v^{\veps}(\bx)$ is bounded from below and the fine mesh size $h$ is small enough, the discrete problem of \eqref{OC_SchGLBBasis_Obj} - \eqref{OC_SchGLBBasis_Cons1} is convex and thus admits a unique solution; see \cite{hou2017sparse, li2017computing} for details.
\end{remark}
	
\subsection{Exponential decay of the multiscale finite element basis functions}
\noindent
It can be proved that the multiscale basis functions $\{\phi_i(\bx)\}_{i=1}^{N_H}$ decay exponentially fast away from its associated vertex $\bx_i\in\mathcal{N}_{c}$ under certain conditions. This allows us to localize the basis functions to a relatively smaller domain and reduce the computational cost. We first define a series of nodal patches $\{D_{\ell}\}$ associated with $\bx_i\in\mathcal{N}$ as
\begin{align}
D_{0}&:=\textrm{supp} \{ \varphi_{i} \} = \cup\{K\in\mathcal{T}_H | \bx_i \in K \}, \label{nodal_patch0}\\
D_{\ell}&:=\cup\{K\in\mathcal{T}_H | K\cap \overline{D_{\ell-1}} \neq \emptyset\}, \quad  \ell=1,2,\cdots.
\label{nodal_patchl}
\end{align}	
\begin{assumption} \label{CoarseMeshResolution}
We assume the potential $v^{\epsilon}(\bx)$ is bounded, i.e.,  $V_0 := ||v^{\epsilon}(\bx)||_{L^\infty (D)}< +\infty$ and the mesh size $H$ of $\mathcal{T}_{H}$ satisfies
\begin{align}
\label{eqn:meshcondition}
\sqrt{V_0}H/\epsilon\lesssim 1,
\end{align}
where $\lesssim$ means bounded from above by a constant.
\end{assumption}
\noindent Under this resolution assumption for the coarse mesh, many typical potentials in the Schr\"{o}dinger equation \eqref{Sch} can be treated as a perturbation to the kinetic operator. Thus, they can be computed using our method. Then, we can show that the multiscale finite element basis functions have the exponentially decaying property.
\begin{proposition}[Exponentially decaying property]\label{ExponentialDecay}
Under the resolution condition of the coarse mesh, i.e., \eqref{eqn:meshcondition}, there exist constants $C>0$ and $0<\beta<1$ independent of $H$, such that
\begin{align}\label{eqn:exponentialdecay}
||\nabla \phi_i(\bx) ||_{L^2(D\backslash D_{\ell})}\leq C \beta^{\ell} ||\nabla \phi_i(\bx) ||_{L^2(D)},
\end{align}
for any $i=1, 2, ..., N_H$.
\end{proposition}
Proof of \eqref{eqn:exponentialdecay} will be given in \cite{ChenMaZhang:prep}. The main idea is to combine an iterative Caccioppoli-type argument \cite{maalqvist2014localization, li2017computing} and some refined estimates with respect to $\veps$. 

The exponential decay of the basis functions enables us to localize the support sets of
the basis functions $\{\phi_i(\bx)\}_{i=1}^{N_H}$, so that the corresponding stiffness matrix is sparse and the
computational cost is reduced. In practice, we define a modified constrained optimization problem as follows
\begin{align}
\phi_i^{\textrm{loc}}& =\underset{\phi \in H_{\textrm{P}}^{1}(D)}{\arg\min} ||\phi||_V   \label{OC_SchBasis_Obj}\\
\text{s.t.}\ &\int_{D_{l^*}}\phi \varphi^{H}_{j} \d\bx= \delta_{i,j},\ \forall 1\leq j \leq N_{H}, \label{OC_SchBasis_Cons1}\\
&\phi(\bx)= 0, \ \bx\in D\backslash D_{l^*},\label{OC_SchBasis_Cons2}
\end{align}
where $D_{l^*}$ is the support set of the localized multiscale basis function $\phi_i^{\textrm{loc}}(\bx)$ and the choice of the integer $l^*$ depends on the decaying speed of $\phi_i^{\textrm{loc}}(\bx)$. In \eqref{OC_SchBasis_Cons1} and \eqref{OC_SchBasis_Cons2}, we have used the fact that $\phi_i(\bx)$ has the exponentially decaying property so that we can localize the support set of $\phi_i(\bx)$ to a smaller domain $D_{l^*}$. In numerical experiments, we find that a small integer $l^*\sim \log(L/H)$  will give accurate results, where $L$ is the diameter of domain $D$. Moreover, the optimization problem \eqref{OC_SchBasis_Obj}-\eqref{OC_SchBasis_Cons2} can be solved in parallel. Therefore, the exponentially decaying property significantly reduces our computational cost in constructing basis functions and computing the solution of the Schr\"{o}dinger equation \eqref{Sch}.

With the localized multiscale finite element basis functions $ \{\phi_i^{\textrm{loc}}(\bx)\}_{i=1}^{N_H}$, we can approximate the wavefunction by $ \psi^\epsilon(\bx,t)=\sum_{i=1}^{N_H}c_i(t)\phi_i^{\textrm{loc}}(\bx) $ using the Galerkin method.

\section{Multiscale reduced basis functions for the random Schr\"{o}dinger equation}\label{sec:RandomCase}
\subsection{Parametrization of the random potential} \label{sec:Parametrization}
\noindent
The random potential $v^{\epsilon}(\bx,\omega)$ is used to model the disorder in a given material. Specifically, we assume $v^{\epsilon}(\bx,\omega)$ is a second order random field, i.e., $v^{\epsilon}(\bx,\omega) \in L^2(D, \Omega)$, with mean $\EEp{v^{\epsilon}(\bx,\omega)}=\bar{v}^{\epsilon}(\bx)$ and covariance kernel $C(\textbf{x},\textbf{y})$. For example, we can choose the covariance kernel as
\begin{align}
C(\textbf{x},\textbf{y}) = \sigma^2 \exp\big(-\sum_{i=1}^{d}\frac{|x_{i}-y_{i}|^2}{2l^2_{i}}\big),
\label{Model_permeability_covkernal}
\end{align}
where $\sigma$ is a constant and $l_{i}$'s are the correlation lengths in each dimension. We also assume that the random potential $v^{\epsilon}(\bx,\omega)$ is almost surely bounded, namely there
exist $v_{\mathrm{max}}$ and $v_{\mathrm{min}}$, such that 
\begin{align}
P(\omega\in \Omega~|~v^{\epsilon}(\bx,\omega)\in [v_{\mathrm{min}},v_{\mathrm{max}}],~ \forall \bx \in D) = 1.
\label{asBoundedPotential}
\end{align}
Circulant embedding method \cite{circulantEmbed1997} and Karhunen-Lo\`{e}ve (KL) expansion method \cite{Karhunen:47,Loeve:78}
are commonly used to generate samples of $v^{\epsilon}(\bx,\omega)$, and the latter will be used in the current work.
The KL expansion of $v^{\epsilon}(\bx,\omega)$ reads as
\begin{align}
v^{\epsilon}(\bx,\omega)= \bar{v}^{\epsilon}(\bx) + \sum_{i=1}^\infty \sqrt{\lambda_i} \xi_i(\omega)v_i(\bx),\label{KLE-poterntial}
\end{align}
where $\xi_i(\omega)$'s are mean-zero and uncorrelated random variables, i.e., $\EEp{\xi_i}=0$, $\EEp{\xi_i\xi_j}=\delta_{ij}$, and $\{\lambda_i,v_i(\bx)\}_{i=1}^{\infty}$ are the eigenpairs of the
covariance kernel $C(\textbf{x},\textbf{y})$. Generally, $\lambda_i$'s
are sorted in a descending order and their decay rates depend on the regularity of the covariance kernel. It has been proven that an algebraic decay rate, i.e.  $\lambda_i=\mathcal{O}(i^{-\gamma})$, is achieved asymptotically if the covariance kernel is of finite Sobolev regularity, and an exponential decay rate is achieved, i.e., $\lambda_i=\mathcal{O}(e^{-\gamma i})$ for some $\gamma >0$, if the covariance kernel is piecewisely analytic \cite{schwab2006karhunen}.  
 
In practice, we truncate the KL expansion \eqref{KLE-poterntial} into its first $m$ terms and obtain a parametrization of the random potential as
\begin{align}
v^{\epsilon}_{m}(\bx,\omega) =  \bar{v}^{\epsilon}(\bx) + \sum_{i=1}^m \sqrt{\lambda_i} \xi_i(\omega)v_i(\bx),  \label{KLE-poterntial2}
\end{align}
which will be used in both analysis and numerics in the remaining part of the paper. 
\begin{remark}
In general, the decay rate of $\lambda_i$ depends on the correlation lengths $l_{i}$, $i=1,...,d$ of the random field $v^{\epsilon}(\bx,\omega)$. Small correlation length results in slow decay of the eigenvalues.
When the correlation lengths approach zero, the random field $v^{\epsilon}(\bx,\omega)$ becomes 
a spatially white noise, which is the case used in the original physics paper \cite{Anderson:58}. 
\end{remark}
   
\subsection{Construction of the multiscale reduced basis functions} \label{sec:BuildMuliStocBasis}
\noindent
For the random Schr\"{o}dinger equation \eqref{eqn:Sch}, it is prohibitively expensive to construct multiscale basis functions for each realization of the random potential using \eqref{OC_SchBasis_Obj} - \eqref{OC_SchBasis_Cons2}. To address this issue, we use a model reduction method to build a small number of reduced basis functions that enable us to obtain multiscale basis functions in a cheaper way without loss of approximation accuracy.

For every $\bx_k\in\mathcal{N}$, we first compute a set of samples of multiscale basis functions 
associated to the vertex $\bx_k$. Specifically, let $\{v^{\epsilon}(\bx,\omega_q)\}_{q=1}^{Q}$ be 
samples of the random potential that are obtained using Monte Carlo (MC) method or qMC method, where $Q$ is the number of samples. Denote $\zeta_{0}^{k}(\bx)=\frac{1}{Q}\sum_{q=1}^{Q}\phi_k^{\textrm{loc}}(\bx,\omega_q)$ the sample mean of the basis functions, and
$\tilde{\phi}_k^{\textrm{loc}}(\bx,\omega_q)=\phi_k^{\textrm{loc}}(\bx,\omega_q)-\zeta_{0}^{k}(\bx)$  is the fluctuation of the $k-$th basis function.

We apply the proper orthogonal decomposition (POD) method \cite{berkooz1993proper,sirovich1987turbulence} to $V=\{\tilde{\phi}_k^{\textrm{loc}}(\bx,\omega_q)\}_{q=1}^{Q}$ and build a set of basis functions 
$\{\zeta^{k}_{1}(\bx),\zeta^{k}_{2}(\bx),...,\zeta^{k}_{m_k}(\bx)\} $ with $m_k \ll Q$ that optimally approximates $V$.
Quantitatively, we have the following approximating property.
\begin{proposition}
\label{POD_proposition} 
Let $\lambda_1 \geq \lambda_2 \geq ... \geq \lambda_{m_k} \geq \lambda_{m_{k+1}} \geq ... > 0$ be positive eigenvalues of the covariance kernel associated with the snapshot of the fluctuations $V$ and the corresponding eigenfunctions are $\zeta_{1}^{k}(\bx)$, ..., $\zeta_{m_k}^{k}(\bx)$,.... Then, the reduced basis functions $\{\zeta_{l}^{k}(\bx)\}_{l=1}^{m_k}$ have the following approximation property
%\begin{align}
%\frac{1}{Q}\sum_{q=1}^{Q}\Big|\Big|\tilde{\phi}_k^{\textrm{loc}}(\bx,\omega_q)
%-\sum_{l=1}^{m_k}\big(\tilde{\phi}_k^{\textrm{loc}}(\bx,\omega_q),\zeta^{k}_l(\bx)\big)_{X}\zeta^{k}_l(\bx)\Big|\Big|_{X}^{2}=\big(\sum_{s=m_k+1}^{Q}\lambda_s\big)\frac{1}{Q}\sum_{q=1}^{Q}\Big|\Big|\tilde{\phi}_k^{\textrm{loc}}(\bx,\omega_q)
%\Big|\Big|_{X}^{2} , %\quad \forall j=1,...Q,
%\label{Prop_PODError}
%\end{align}
\begin{align}
\frac{\sum_{q=1}^{Q}\Big|\Big|\tilde{\phi}_k^{\textrm{loc}}(\bx,\omega_q)
	-\sum_{l=1}^{m_k}\big(\tilde{\phi}_k^{\textrm{loc}}(\bx,\omega_q),\zeta^{k}_l(\bx)\big)_{X}\zeta^{k}_l(\bx)\Big|\Big|_{X}^{2}  }{ \sum_{q=1}^{Q}\Big|\Big|\tilde{\phi}_k^{\textrm{loc}}(\bx,\omega_q)
	\Big|\Big|_{X}^{2}  }
 = \frac{\sum_{s=m_k+1}^{Q}\lambda_s}{\sum_{s=1}^{Q}\lambda_s}  ,  \label{Prop_PODError}
\end{align}
where $ X=L^2(D) $ or $ X=H^1(D) $ and the number $m_k$ is determined according to the 
ratio $\rho=\frac{\sum_{s=1}^{m_k}\lambda _{s}}{\sum_{s=1}^{Q}\lambda_{s}}$. 	
\end{proposition} 
In practice, we choose the first $m_k$ dominant reduced basis functions such that $\rho$ is close 
enough to $1$ to achieve a desired accuracy, say $\rho=99\%$. More details of the POD method can be found in \cite{berkooz1993proper,sirovich1987turbulence}. 
Notice that reduced basis functions $\zeta^{k}_{0}(\bx)$ and $\zeta^{k}_{l}(\bx)$, $l=1,...,m_k$ approximately capture the mean profile and the fluctuation of multiscale basis functions associated with $\bx_{k}$, respectively. Thus, it is expected that for each realization of the random potential the associated multiscale basis functions can be approximated by the reduced basis functions, i.e., 
\begin{align}
\phi_k^{\textrm{loc}}(\bx,\omega)\approx\zeta^{k}_{0}(\bx)+\sum_{l=1}^{m_{k}}c_{l}(\omega)\zeta^{k}_{l}(\bx). \label{RB_expansion}
\end{align}

\begin{remark}
To construct the multiscale reduced basis functions, we partition the coarse grids $D^k$ into fine-scale quadrilateral elements with meshsize $h\ll \varepsilon$, which requires additional computational cost in the offline stage. However, the precomputed reduced basis functions can be used repeatedly to solve \eqref{eqn:Sch} for each realization of the random potential and different initial data, which results in considerable savings.
\end{remark}

\subsection{Estimation of the number of learning samples} \label{sec:DetermineSamples}
\noindent
We shall study the continuous dependence of multiscale basis functions on the random potential, which provide a guidence on how to determine the number of
samples in the construction of multiscale basis functions. For notational simplification, we carry out the analysis for multiscale basis functions without localization.

Let $\varphi^{h}_{s}(\bx)$, $s=1,...,N_h$ denote the finite element basis functions defined on fine mesh with size $h$ and $N_h$ is the number of fine-scale finite element basis functions. When we numerically solve \eqref{OC_SchGLBBasis_Obj}-\eqref{OC_SchGLBBasis_Cons1}, we represent the multiscale basis function as $\phi_i(\bx)=\sum_{s=1}^{N_h}c_{s}\varphi^{h}_{s}(\bx)$ and obtain the following quadratic programming problem with equality constraints
\begin{equation}
\left\{
\begin{aligned}
& \min_{\boldsymbol{c}} \frac{1}{2}\boldsymbol{c}^T Q \boldsymbol{c},\\
& \textrm{s.t. } A\boldsymbol{c}=\boldsymbol{b},
\end{aligned}
\right.
\label{eqn:QP}
\end{equation}
where $\boldsymbol{c} = [c_{1},...,c_{N_h}]^{T}$ is the coefficients and $Q$ is a symmetric positive definite matrix on the fine triangularization $\mathcal{T}_h$ with the $(i,j)$ component
\begin{align}\label{eqn:Qij}
Q_{ij}=\frac{\veps^2}{2}(\nabla\varphi^h_{i},\nabla\varphi^h_{j})+(v^{\veps}(\bx,\omega_q)\varphi^h_i,\varphi^h_{j}).
\end{align}
In \eqref{eqn:QP}, $A$ is an $N_h$-by-$N_H$ matrix with  $A_{ij}=(\varphi^h_i,\varphi^H_j)$ 
and $\boldsymbol{b}$ an $N_h$-by-$1$ vector with only the $i-$th entry being $1$ and others being $0$.

The following result states the continuous dependence of multiscale basis functions on the random potential.
\begin{theorem}\label{basis-dependon-potential}
Assume the random potential $v^{\epsilon}(\bx,\omega)$ is almost surely bounded, i.e. \eqref{asBoundedPotential} is satisfied and mesh size of the fine-scale triangles is small such that: (1) $h/\veps=\kappa $ is small; and (2) $h^d  \lVert v^{\veps}(\cdot,\omega_1) - v^{\veps}(\cdot,\omega_2) \rVert_{L^{\infty}(D)}<1$. Then for two realizations $\omega_1$ and $\omega_2$ of the random potential $v^{\veps}(\bx,\omega)$, the corresponding multiscale basis functions satisfy
\begin{align}
\lVert \phi(\cdot,\omega_1) - \phi(\cdot,\omega_2) \rVert_{L^{\infty}(D)} \leq \frac{C}{\kappa^6} \veps^{-2} \lVert v^{\veps}(\cdot,\omega_1) - v^{\veps}(\cdot,\omega_2) \rVert_{L^{\infty}(D)},
\end{align}
where the constant $C$ is independent of $h, \veps$, and $\lVert v^{\veps}(\cdot,\omega_1) - v^{\veps}(\cdot,\omega_2) \rVert_{L^{\infty}(D)}$.
\end{theorem}

\begin{proof}
Under the assumptions that $v^{\epsilon}(\bx,\omega)$ is almost surely bounded and $h/\veps=\kappa $ is small,
we know that $Q$ is a positive definite matrix. Moreover, we know that $A$ has full rank, i.e., $\rank(A)=N_H$. Therefore, the quadratic optimization problem \eqref{eqn:QP} has a unique minimizer, satisfying the Karush-Kuhn-Tucker condition. Specifically, the unique minimizer of \eqref{eqn:QP} can be explicitly written as
\begin{align}\label{eqn:minref}
\boldsymbol{c} = Q^{-1} A^T (AQ^{-1} A^T)^{-1}\boldsymbol{b}.
\end{align}
For two realizations $\omega_1$ and $\omega_2$, we define $\delta V = Q_1- Q_2$. Then
\begin{align}
\left(\delta V\right)_{ij}= \left(\left(v(\cdot,\omega_1)-v(\cdot,\omega_2)\right)\varphi^h_i,\varphi^h_j\right),
\end{align}
and thus
\begin{align}
\lVert \delta V \rVert_{\infty} \leq h^d  \lVert v^{\veps}(\cdot,\omega_1) - v^{\veps}(\cdot,\omega_2) \rVert_{L^{\infty}(D)}.
\end{align}
We choose $h$ to be small enough such that $\lVert \delta V \rVert_{\infty} \leq 1$, and have
\[
Q_2^{-1} = \sum_{n=0}^{\infty} \left(Q_1^{-1}\delta V\right)^n Q_1^{-1},
\]
and thus
\begin{align*}
\boldsymbol{c}_2 - \boldsymbol{c}_1 & = \left[ Q^{-1}_2 - Q^{-1}_1 \right] A^T (AQ^{-1}_1 A^T)^{-1}\boldsymbol{b} 
+ Q^{-1}_2 A^T \left[ (AQ^{-1}_2 A^T)^{-1} - (AQ^{-1}_1 A^T)^{-1}\right]\boldsymbol{b}, \nonumber\\
& = Q^{-1}_1 \delta V Q^{-1}_1  A^T (AQ^{-1}_1 A^T)^{-1}\boldsymbol{b}  \nonumber \\
& \quad - Q^{-1}_2 A^T (AQ^{-1}_1 A^T)^{-1} (AQ^{-1}_1 \delta V Q^{-1}_1 A^T) (AQ^{-1}_1 A^T)^{-1}\boldsymbol{b} + o(\lVert\delta V\rVert_{\infty}),\nonumber \\
& = Q^{-1}_1 \delta V Q^{-1}_1  A^T (AQ^{-1}_1 A^T)^{-1}\boldsymbol{b} \nonumber \\
& \quad - Q^{-1}_1 A^T (AQ^{-1}_1 A^T)^{-1} (AQ^{-1}_1 \delta V Q^{-1}_1 A^T) (AQ^{-1}_1 A^T)^{-1}\boldsymbol{b} + o(\lVert\delta V\rVert_{\infty}).\nonumber 
\end{align*}
Therefore,
\begin{align*}
\lvert \boldsymbol{c}_2 - \boldsymbol{c}_1 \rvert_{\infty} & \leq 
C \lVert A \rVert_{\infty} \lVert Q^{-1}_1 \rVert_{\infty}^2 
\lVert (AQ^{-1}_1A^T)^{-1} \rVert_{\infty}
\lvert \boldsymbol{b} \rvert_{\infty} \left( 1 + 
\lVert A \rVert_{\infty}^2 \lVert Q^{-1}_1 \rVert_{\infty} \lVert (AQ^{-1}_1A^T)^{-1} \rVert_{\infty} \right) \lVert \delta V \rVert_{\infty} .\nonumber 
\end{align*}
By their definitions, we have
\[
\lVert A \rVert_{\infty} \leq C h^d , \quad \lvert \boldsymbol{b} \rvert_{\infty} =1, \quad \lVert Q^{-1}_1 \rVert_{\infty} \leq C h^{-2}, \quad \lVert Q_1 \rVert_{\infty} \leq C \max\{\veps^2,h^2\}\leq C\veps^2,
\]
and thus
\begin{align*}
\lvert \boldsymbol{c}_2 - \boldsymbol{c}_1 \rvert_{\infty} & \leq 
C \veps^{4} h^{-6} h^{-d} \lVert \delta V \rVert_{\infty}
\leq C \veps^{4} h^{-6} \lVert v^{\veps}(\cdot,\omega_2) - v^{\veps}(\cdot,\omega_1) \rVert_{L^{\infty}(D)}.\nonumber 
\end{align*}
We complete the proof since $h/\veps=\kappa $ and $\lVert \phi(\cdot,\omega_2) - \phi(\cdot,\omega_1) \rVert_{L^{\infty}(D)} \leq \lvert \boldsymbol{c}_2 - \boldsymbol{c}_1 \rvert_{\infty}$.
\end{proof}	
Equipped with \textbf{Theorem \ref{basis-dependon-potential}}, we can estimate the number of samples in the construction of multiscale reduced basis functions. Suppose the random potential is of the form \eqref{KLE-poterntial2}.  
For any $ \delta > 0$, we choose an integer $Q_{\delta}$ and a set of random samples $\{v^{\epsilon}(\bx,\omega_q)\}_{q=1}^{Q_{\delta}}$ such that
\begin{align} 
\mathds{E}\left[\inf_{1\le q\le Q_{\delta}} \big|\big|v_{m}^{\epsilon}(\bx,\omega) - v_{m}^{\epsilon}(\bx,\omega_q)\big|\big|_{L^\infty(D)}\right] \le \delta,  \label{asd}
\end{align} 
where the expectation is taken over the random variables in $v^{\epsilon}_{m}(\bx,\omega)$ of the form \eqref{KLE-poterntial2}. We can give a way to choose the random samples $\{v^{\epsilon}(\bx,\omega_q)\}_{q=1}^{Q_{\delta}}$ since the distribution of the random variables $\xi_{i}(\omega)$, $i=1,...,m$ is known.

For every $\bx_k\in\mathcal{N}$, let $\{\phi_k(\bx,\omega_q)\}_{q=1}^{Q_{\delta}}$ be the samples of multiscale basis functions associated with $\bx_k$. Then, we have  
\begin{align} 
\mathds{E}\left[\inf_{1\le q\le Q_{\delta}} \big|\big|\phi_k(\bx,\omega) - \phi_k(\bx,\omega_q)\big|\big|_{L^\infty(D)}\right] \le \frac{C}{\kappa^6} \veps^{-2}\delta.  \label{basis-potential}
\end{align}
Given parameters $\veps$ and $h$, we choose $\delta$ and $Q_{\delta}$ so that the right-hand side of \eqref{basis-potential} is small. Then the space of multiscale basis functions can be well approximated by the samples of multiscale basis functions $\{\phi_k(\bx,\omega_q)\}_{q=1}^{Q_{\delta}}$ with
controllable accuracy and the POD method is further applied to construct multiscale reduced basis functions.

\subsection{Derivation of our method based on the multiscale reduced basis functions}
\noindent  
In this section, we present our method for solving the random Schr\"{o}dinger equation: in the physical space, we use the multiscale reduced basis functions obtained in \secref{sec:BuildMuliStocBasis}; in the random space, we use the qMC method. 

The implementation of the qMC method is fairly easy. For instance, given a set of qMC samples, expectation of the solution is approximated by  
\begin{align}\label{qmc-compute-mean}
\mathds{E}\left[\psi^\epsilon(t,\bx,\omega)\right]\approx \frac{1}{n}\sum_{i=1}^{n}\psi^\epsilon(t,\bx,\omega_i),
\end{align}
where $n$ is the number of qMC samples. Details of the generation of qMC samples and its convergence analysis will be discussed in \secref{sec:analysis}.

Now, we focus on how to approximate the wavefunction in the physical space for each qMC sample $\omega_s$. 
For each node point $\bx_k\in\mathcal{N}$, we have constructed a set of multiscale reduced basis functions $\{\zeta^{k}_{i}\}_{i=0}^{m_{k}}$
and represent the wavefunction by
\begin{align}
\psi^\epsilon(t,\bx,\omega_s)=\sum_{k=1}^{N_H}  \sum_{l=0}^{m_{k}}c^{k}_{l}(t,\omega_s)\zeta^{k}_{l}(\bx),
\label{POD-basis-representation}
\end{align}
where $m_k$ is the number of multiscale reduced basis functions associated with the node $\bx_k$.   
In the Galerkin formulation, we have the following weak form
\begin{align}
&\left(i\epsilon\partial_t\sum_{k=1}^{N_H}  \sum_{l=0}^{m_{k}}c^{k}_{l}(t,\omega_s)\zeta^{k}_{l}(\bx),\zeta^{j}_{r}(\bx)\right)=\left(\mathcal{H}(\bx,\omega_s)\sum_{k=1}^{N_H}  \sum_{l=0}^{m_{k}}c^{k}_{l}(t,\omega_s)\zeta^{k}_{l}(\bx),\zeta^{j}_{r}(\bx)\right),\nonumber\\
&\quad \bx\in D,\quad t\in \mathbb{R},\quad j=1,\cdots,N_H, \quad r=0,\cdots,m_{k},
\label{appro_Sch_random}
\end{align}
where $\mathcal{H}(\bx,\omega_s) $ is a deterministic operator. To numerically solve \eqref{appro_Sch_random}, we introduce some notations. Let $S$, $M$, and $V(\omega_s)$ be matrices with dimension $ \sum_{k=1}^{N_H}(m_k+1)\times \sum_{k=1}^{N_H}(m_k+1) $. Their entries are given by 
\begin{align*}
S_{\sum_{i=1}^{k}(m_i+1)+l,~\sum_{i=1}^{j}(m_i+1)+r}&=\int_{D} \nabla \zeta^{k}_{l} \cdot \nabla \zeta^{j}_{r} \d\bx, \\
M_{\sum_{i=1}^{k}(m_i+1)+l,~\sum_{i=1}^{j}(m_i+1)+r}&=\int_{D} \zeta^{k}_{l} \zeta^{j}_{r} \d\bx, \\
V_{\sum_{i=1}^{k}(m_i+1)+l,~\sum_{i=1}^{j}(m_i+1)+r}(\omega_s)&=\int_{D} \zeta^{k}_{l} v^{\epsilon}(\bx,\omega_s) \zeta^{j}_{r}\d\bx.
\end{align*}
%(\textcolor{red}{??? what is the size of those matrices.}\textcolor{blue}{The matrices size are $ \sum_{k=1}^{N_H}(m_k+1)\times \sum_{k=1}^{N_H}(m_k+1) $}.)
%\begin{align*}
%S^{k,j}_{l,r}&=\int_{D} \nabla \zeta^{k}_{l} \cdot \nabla \zeta^{j}_{r} \d\bx, \\
%M^{k,j}_{l,r}&=\int_{D} \zeta^{k}_{l} \zeta^{j}_{r} \d\bx, \\
%V^{k,j}_{l,r}(\omega_s)&=\int_{D} \zeta^{k}_{l} v^{\epsilon}(\bx,\omega_s) \zeta^{j}_{r}\d\bx.
%\end{align*}
Then, we can reduce the weak formulation \eqref{appro_Sch_random} into the following ODE system 
\begin{align}\label{eqn:Schmatrix}
i\epsilon M\frac{\mathrm{d}\bc(t,\omega_s)}{\mathrm{d}t}=\left(\frac{\epsilon^2}{2}S+V(\omega_s)\right)\bc(t,\omega_s),
\end{align}
where the column vector $ \bc(t,\omega_s)=(c^{1}_{0}(t,\omega_s),...,c^{1}_{m_k}(t,\omega_s),...,c^{N_H}_{0}(t,\omega_s),...,c^{N_H}_{m_k}(t,\omega_s))^T $ consisting of all expansion coefficients of the solution $\psi^\epsilon(t,\bx,\omega_s)$ onto multiscale reduced basis functions. We can further rewrite \eqref{eqn:Schmatrix} as
\begin{align}\label{eqn:Schcompact}
\frac{d\bc(t,\omega_s)}{dt}=\frac{1}{i\epsilon} B(\omega_s)\bc(t,\omega_s)
\end{align}
where $ B(\omega_s)=M^{-1}A(\omega_s) $ and $A(\omega_s)=\frac{\epsilon^2}{2}S+V(\omega_s)$. 
In the end, we can solve the above ODE system using existing ODE solvers. 
 
Before ending this section, we shall explain why we choose the qMC method to approximate the random space of the electron wavefunction.  
Since the parameterization of a random potential may have high dimension, i.e., $m$ is large in \eqref{KLE-poterntial}, non-intrusive methods, 
such as sparse grid method \cite{bungartz2004sparse} and stochastic collocation method \cite{nobile2008sparse}, become prohibitively expensive to solve PDEs with random coefficients. Polynomial chaos expansion (PCE) methods \cite{ghanem2003stochastic,xiu2003modeling} are also frequently used in the literature to solve PDEs with random coefficients. 
This type of methods is useful if the solution is sufficiently smooth in the random space with small dimensionality.
The performance of MC method does not depend on the dimension of the random space. However, its convergence rate is merely $O(\frac{1}{\sqrt{n}})$. The convergence rate of the qMC method is better both theoretically and numerically; see \eqref{qMC-error} in Theorem \ref{main-result}.
Therefore, we choose the qMC method and its implementation is almost the same as the MC method. 
 
\section{Convergence analysis} \label{sec:analysis}
\noindent 	
We shall analyze the approximation error of the proposed method, where the emphasis is put on computing functionals of the wavefunction.

\subsection{Regularity of the wavefunction with respect to the random variables} \label{sec:RegularityInRamdonSpace}
\noindent 	
Since the potential $v^{\epsilon}(\bx,\omega)$ in \eqref{eqn:Sch} is parametrized by $m$ random variables $\xi_i(\omega)$, $i=1,...,m$ in \eqref{KLE-poterntial2}, i.e., $v_{m}^{\epsilon}(\bx,\omega)=v^{\epsilon}(\bx,\xi_1(\omega),...,\xi_m(\omega))$. 
The wavefunction $\psi^\epsilon_m(t,\bx,\omega)$ satisfies 
\begin{equation}\label{Schm}
\left\{
\begin{aligned}
i\epsilon\partial_t\psi_m^\epsilon&=-\frac{\epsilon^2}{2}\Delta\psi_m^\epsilon+v^{\epsilon}_m(\bx,\omega) \psi_m^\epsilon,\quad \bx\in D,\quad t\in \mathbb{R},\\
\psi_m^\epsilon &\in H_{\textrm{P}}^{1}(D),\\
\psi_m^\epsilon|_{t=0}&=\psi_{\textrm{in}}(\bx).
\end{aligned}
\right.
\end{equation} 
The Doob-Dynkin's lemma implies the wavefunction $\psi^\epsilon_m(t,\bx,\omega)$ in \eqref{Schm} can also be represented by a functional of these random variables, i.e., $\psi^\epsilon(t,\bx,\omega)=\psi^\epsilon(t,\bx,\xi_1(\omega),...,\xi_m(\omega))$.

First of all, we analyze the error introduced by the parameterization of the random potential. We have the following estimate result.  
\begin{lemma}\label{lemma:KLerror}
The difference between wavefunctions to \eqref{Schm} and \eqref{eqn:Sch} satisfies
\begin{equation}\label{eqn:KLEpsierror}
\lVert \psi_m^{\veps}-\psi^{\veps}\rVert_{L^2(\Omega, D)} \leq \frac{T }{\veps}\lVert v_m^{\veps}-v^{\veps}\rVert_{L^{\infty}(\Omega, D)}, \quad \forall t\in [0,T].
\end{equation}
\end{lemma}
\begin{proof}
The difference $\delta\psi = \psi_m^{\veps}-\psi^{\veps}$ satisfies
\begin{equation*}
\left\{
\begin{aligned}
i\epsilon\partial_t\delta\psi&=-\frac{\epsilon^2}{2}\Delta\delta\psi
+v_m^{\epsilon} \delta\psi + (v_m^{\veps}-v^{\veps})\psi^{\veps},\quad \bx\in D,\quad t\in \mathbb{R},\\
\delta\psi &\in H_{\textrm{P}}^{1}(D),\\
\delta\psi|_{t=0}&=0.
\end{aligned}
\right.
\end{equation*}
By a direct calculation, we have
\begin{equation*}
	\frac{\mathrm{d}}{\mathrm{d}t}\lVert\delta\psi\rVert_{L^2(\Omega, D)}^2= \frac{1}{i\veps}\int_{\Omega}\int_{D} \left(\overline{\delta\psi}(v_m^{\veps}-v^{\veps})\psi^{\veps}
	-\overline{\psi^{\veps}}(v_m^{\veps}-v^{\veps})\delta\psi\right)
	\mathrm{d}\bx\mathrm{d}\rho(\omega),
	\end{equation*}
where $\rho(\omega)$ is the probability measure induced by the randomness in the potential \eqref{KLE-poterntial2} and thus
%\begin{align*}
%\frac{\mathrm{d}}{\mathrm{d}t}\lVert\delta\psi\rVert_{L^2(\Omega, D)}^2& \leq \frac{2}{\veps}\int_{\Omega}\int_{D} \lvert \overline{\delta\psi}(v_m^{\veps}-v^{\veps})\psi^{\veps}\rvert\mathrm{d}\bx\mathrm{d}\rho(\omega)
%\leq \frac{2}{\veps}\int_{\Omega} \lVert \overline{\delta\psi}\rVert_{L^2(D)}\lVert v_m^{\veps}-v^{\veps})\psi^{\veps}\rVert_{L^2(D)}\mathrm{d}\rho(\omega), \\
%& \leq \frac{2}{\veps}\int_{\Omega} \lVert \delta\psi\rVert_{L^2(D)}\lVert v_m^{\veps}-v^{\veps}\rVert_{L^{\infty}(D)}\mathrm{d}\rho(\omega)
%\leq \frac{2\lVert v_m^{\veps}-v^{\veps}\rVert_{L^{\infty}(D,\Omega)}}{\veps}\int_{\Omega} \lVert \delta\psi\rVert_{L^2(D)}\mathrm{d}\rho(\omega), \\
%& \leq \frac{2\lvert \Omega\rvert^{1/2} \lVert v_m^{\veps}-v^{\veps}\rVert_{L^{\infty}(D,\Omega)}}{\veps} \lVert \delta\psi\rVert_{L^2(\Omega, D)}.
%\end{align*}
\begin{align*}
\frac{\mathrm{d}}{\mathrm{d}t}\lVert\delta\psi\rVert_{L^2(\Omega, D)}^2& \leq \frac{2}{\veps}\int_{\Omega}\int_{D} \lvert \overline{\delta\psi}(v_m^{\veps}-v^{\veps})\psi^{\veps}\rvert\mathrm{d}\bx\mathrm{d}\rho(\omega)
\leq \frac{2}{\veps}\int_{\Omega} \lVert \overline{\delta\psi}\rVert_{L^2(D)}\lVert v_m^{\veps}-v^{\veps})\psi^{\veps}\rVert_{L^2(D)}\mathrm{d}\rho(\omega), \\
& \leq \frac{2}{\veps}\int_{\Omega} \lVert \delta\psi\rVert_{L^2(D)}\lVert v_m^{\veps}-v^{\veps}\rVert_{L^{\infty}(D)}\mathrm{d}\rho(\omega)
\leq \frac{2\lVert v_m^{\veps}-v^{\veps}\rVert_{L^{\infty}(D,\Omega)}}{\veps} \lVert \delta\psi\rVert_{L^2(D,\Omega)}.
\end{align*}
Therefore, we obtain
\begin{align*}
\lVert\delta\psi\rVert_{L^2(\Omega, D)} \leq \frac{T}{\veps}\lVert v_m^{\veps}-v^{\veps}\rVert_{L^{\infty}(\Omega, D)}, \quad \forall t\in [0,T],
\end{align*}
which completes the proof.
\end{proof}

To analyze the qMC method, it is crucial to bound the mixed first derivatives of $\psi_m^\epsilon$ with respect to $\xi_i(\omega)$.
Denote $\boldsymbol{\xi}(\omega)=(\xi_1(\omega),\cdots,\xi_m(\omega))^T$ for convenience.
Let $\boldsymbol{\nu} = (\nu_1,\cdots,\nu_m)$ denote a multi-index of non-negative integers, with $\lvert\boldsymbol{\nu}\rvert = \sum_{j = 1}^m \nu_j$
and $\lvert\boldsymbol{\nu}\rvert_{\infty} = \max_{1\leq j\leq m} \nu_j$
. The value of $\nu_j$ determines the number of derivatives to be taken with respect to $\xi_j$,
and $\partial^{\boldsymbol{\nu}}\psi_m^{\veps}$ denotes the mixed derivative of $\psi_m^{\veps}$ with respect to all variables specified by the multi-index $\boldsymbol{\nu}$. 

\begin{lemma}\label{lemma:qmc1}
For any $\omega\in \Omega$, any time $T$, and for any multi-index $\boldsymbol{\nu}$ with $\lvert\boldsymbol{\nu}\rvert < \infty$, 
the partial derivative of $\psi_m^{\veps}(t, \bx, \omega)$ satisfies the following a-priori estimate
\begin{align}\label{eqn:estimate1}
	\lVert\partial^{\boldsymbol{\nu}}\psi_m^{\veps}(t,\cdot,\omega)\rVert_{L^2(D)}\leq \frac{ \lvert \boldsymbol{\nu}\rvert!\;T^{\lvert \boldsymbol{\nu}\rvert}}{\veps^{\lvert \boldsymbol{\nu}\rvert}} \left\{\prod_{\substack{j\ge 1}} \left(\sqrt{\lambda_j}\lVert v_j\rVert_{C^0(\bar{D})}\right)^{\nu_j} \right\},\quad\forall\; t\in [0,T].
	\end{align}
\end{lemma}

\begin{proof}
	When $\lvert\boldsymbol{\nu}\rvert = 1$, we take the derivative of \eqref{Schm} with respect to $\xi_j(\omega)$. Let
	$\partial_j\psi_m = \partial_{\xi_j} \psi_m^{\veps}$ and $\partial_j v_m = \partial_{\xi_j} v_m^{\veps}$, we have
	\begin{equation}\label{eqn:SchwrtRv}
	i\veps \left(\partial_j\psi_m \right)_t = -\frac{\veps^2}{2}\Delta \left(\partial_j\psi_m \right) + 
	\left(\partial_j v_m \right) \psi_m^{\veps} + v_m^{\veps} \left(\partial_j\psi_m\right).\nonumber
	\end{equation}
	Thereafter, we have the following estimate by a direction calculation
	\begin{align}
	\frac{\mathrm{d}}{\mathrm{d}t}\lVert \partial_j\psi_m \rVert^2_{L^2(D)} & = \int _{D} \left\{\left(\overline{\partial_j\psi_m}\right)_t\left(\partial_j\psi_m\right) + 
	\left(\overline{\partial_j\psi_m}\right)\left(\partial_j\psi_m\right)_t\right\}\mathrm{d} \bx, \nonumber\\
	& = \int _{D} \left(-\frac{1}{i\veps}\left(\partial_j v_m\right)\overline{\psi_m^{\veps}}\left(\partial_j\psi_m\right) + \frac{1}{i\veps} \left(\overline{\partial_j\psi_m}\right) \left(\partial_j v\right)\psi_m^{\veps}\right)\mathrm{d} \bx, \nonumber\\
	& \leq \frac{2}{\veps} \lVert \partial_j\psi_m \rVert_{L^2(D)}  \lVert \partial_j v \psi_m^{\veps} \rVert_{L^2(D)}
	\leq \frac{2}{\veps} \lVert \partial_j\psi_m \rVert_{L^2(D)}  \lVert \partial_j v_m  \rVert_{L^\infty(D)},\nonumber
	\end{align}
	and
	\begin{align}\label{eqn:regrandom1}
	\lVert \partial_j \psi_m \rVert_{L^2(D)}\leq \frac{T}{\veps}\lVert \partial_j v_m  \rVert_{L^\infty(D)} \leq \frac{T}{\veps}\sqrt{\lambda_j}\lVert \phi_j  \rVert_{C^0(\bar{D})}.
	\end{align}
	
	When $\lvert\boldsymbol{\nu}\rvert \ge 2$, we have
	\begin{align}
	i\veps \left(\partial^{\boldsymbol{\nu}} \psi_m \right)_t = -\frac{\veps^2}{2}\Delta \left(\partial^{\boldsymbol{\nu}}\psi_m \right) + 
	\sum_{\substack{\boldsymbol{\mu}\preceq \boldsymbol{\nu} \\ \boldsymbol{\mu}\neq \boldsymbol{\nu}}} \binom{\boldsymbol{\nu}}{\boldsymbol{\mu}}\left(\partial^{\boldsymbol{\nu}-\boldsymbol{\mu}} v_m \right) \left(\partial^{\boldsymbol{\mu}}\psi_m\right) + v_m^{\veps} \left(\partial^{\boldsymbol{\nu}}\psi_m\right).\nonumber
	\end{align}
	According to the definition of the random potential \eqref{KLE-poterntial}, we have $\partial^{\boldsymbol{\nu}-\boldsymbol{\mu}} v_m^{\veps} = 0$ if $\lvert \boldsymbol{\nu}-\boldsymbol{\mu}\rvert \ge 2$. Thus the above equation can be simplified as
	\begin{align}
	i\veps \left(\partial^{\boldsymbol{\nu}} \psi_m \right)_t = -\frac{\veps^2}{2}\Delta \left(\partial^{\boldsymbol{\nu}}\psi_m \right) + 
	\sum_{\lvert\boldsymbol{\nu}-\boldsymbol{\mu}\rvert = 1} \binom{\lvert\boldsymbol{\nu}\rvert}{1}
	\left(\partial^{\boldsymbol{\nu}-\boldsymbol{\mu}} v_m \right) \left(\partial^{\boldsymbol{\mu}}\psi_m\right) + v_m^{\veps} \left(\partial^{\boldsymbol{\nu}}\psi_m\right).\nonumber
	\end{align}
	
	Similarly, we obtain
	\begin{align}
	\frac{\mathrm{d}}{\mathrm{d}t}\lVert \partial^{\boldsymbol{\nu}} \psi_m \rVert^2_{L^2(D)} & = \int _{D} \left\{\left(\overline{\partial^{\boldsymbol{\nu}}\psi_m}\right)_t\left(\partial^{\boldsymbol{\nu}}\psi_m\right) + 
	\left(\overline{\partial^{\boldsymbol{\nu}}\psi_m}\right)\left(\partial^{\boldsymbol{\nu}}\psi_m\right)_t\right\}\mathrm{d} \bx, \nonumber\\
	& = \sum_{\lvert\boldsymbol{\nu}-\boldsymbol{\mu}\rvert = 1} \binom{\lvert\boldsymbol{\nu}\rvert}{1} \int _{D} \left(-\frac{1}{i\veps}
	\left(\partial^{\boldsymbol{\nu}-\boldsymbol{\mu}} v_m \right) \overline{\left(\partial^{\boldsymbol{\mu}}\psi_m\right)}\left(\partial^{\boldsymbol{\nu}}\psi_m\right) + \frac{1}{i\veps}  \left(\overline{\partial^{\boldsymbol{\nu}}\psi_m}\right)
	\left(\partial^{\boldsymbol{\nu}-\boldsymbol{\mu}} v_m \right) \left(\partial^{\boldsymbol{\mu}}\psi_m\right)\right)\mathrm{d} \bx, \nonumber\\
	& \leq \frac{2\lvert \boldsymbol{\nu}\rvert}{\veps} \lVert \partial^{\boldsymbol{\nu}}\psi_m \rVert_{L^2(D)}  
	\sum_{\lvert\boldsymbol{\nu}-\boldsymbol{\mu}\rvert = 1}
	\lVert \left(\partial^{\boldsymbol{\nu}-\boldsymbol{\mu}} v_m \right)\rVert_{L^{\infty}(D)}\lVert \left(\partial^{\boldsymbol{\mu}}\psi_m\right) \rVert_{L^2(D)}, \nonumber
	\end{align}
	and
	\begin{align}\label{eqn:regrandom2}
	\lVert \partial^{\boldsymbol{\nu}} \psi_m \rVert_{L^2(D)} 
	\leq \frac{T\lvert \boldsymbol{\nu}\rvert}{\veps} 
	\sum_{\lvert\boldsymbol{\nu}-\boldsymbol{\mu}\rvert = 1}
	\lVert \left(\partial^{\boldsymbol{\nu}-\boldsymbol{\mu}} v_m \right)\rVert_{L^{\infty}(D)}\lVert \left(\partial^{\boldsymbol{\mu}}\psi_m\right) \rVert_{L^2(D)}.
	\end{align}
Now we are ready to prove the theorem by mathematical induction. From \eqref{eqn:regrandom1}, we know that \eqref{eqn:estimate1} holds for $\lvert \boldsymbol{\nu}\rvert = 1$. Assume that \eqref{eqn:estimate1} holds for $\boldsymbol{\mu}$ with $\lvert \boldsymbol{\nu}-\boldsymbol{\mu}\rvert = 1$. Substituting this into \eqref{eqn:regrandom2} yields the desired estimate for the $\boldsymbol{\nu}$ case. 
\end{proof}

\begin{remark}
The above derivation is similar to that in \cite{jin2019gaussian}, where an estimate in $L^2(D,\Omega)$ norm is obtained. Here, for each random realization $\omega$, we have the esitmate \eqref{eqn:estimate1} in $L^2(D)$ norm, which will be used to prove the convergence in qMC.
\end{remark}

\subsection{Main result of the error analysis} \label{sec:erroranalysis}
\noindent  
In the framework of uncertainty quantification, we are interested in computing some statistical quantities of the electron wavefunction. 
As such, we shall present the error analysis of our method in computing functionals of $\psi_m^{\veps}$.

Let $\mathcal{G}(\cdot)$ be a continuous linear functional on $L^2(D)$, then there exists a constant $C_{\mathcal{G}}$ such that
\[
\lvert \mathcal{G}(u) \rvert \leq C_{\mathcal{G}} \lVert u\rVert_{L^2(D)},
\]
for all $u\in L^2(D)$. Consider the following integral
\begin{equation}\label{eqn:integral_highdim}
I_m(F) = \int_{\boldsymbol{\xi}\in [0,1]^m} F(\boldsymbol{\xi})\mathrm{d}\boldsymbol{\xi}
\end{equation}
with $F(\boldsymbol{\xi})=\mathcal{G}(\psi_m^{\veps}(\cdot,\boldsymbol{\xi}))$.
We approximate the integral over the unit cube by randomly shifted lattice
rules
\[
Q_{m,n}(\boldsymbol{\Delta};F)\triangleq \frac{1}{n}\sum_{i=1}^{n} F\big(\textrm{frac}(\frac{i\boldsymbol{z}}{n}+\boldsymbol{\Delta})\big),
\]
where $\boldsymbol{z}\in\mathbb{N}^m$ is the (deterministic) generating vector
and $\boldsymbol{\Delta}\in [0,1]^m$ is the random shift which is uniformly
distributed over $[0,1]^m$. Notice that $m$ is the dimension of the random vector $\boldsymbol{\xi}$ 
in the random potential and $n$ is the number of the sample point in implementing the qMC method.
The interested reader is referred to \cite{dick2013high} for more details of the randomly shifted lattice rules in the qMC method. 
\begin{lemma}\label{lemma:qmc2}
Let $F$ be the integrand in \eqref{eqn:integral_highdim}. Given $m,n\in\mathbb{N}$ with $n\le 10^{30}$, weights $\boldsymbol{\gamma}=(\gamma_{\boldsymbol{\mathfrak{u}}})_{\boldsymbol{\mathfrak{u}}\subset \mathbb{N}}$, a randomly shifted lattice rule with $n$ points in $m$ dimensions can be constructed by a component-by-component algorithm such
	that, for all $\lambda\in(1/2,1]$,
	\begin{equation}\label{eqn:qmcerror}
	\sqrt{\mathbb{E}^{\boldsymbol{\Delta}}\lvert I_m(F)-Q_{m,n}(\cdot; F)\rvert^2}
	\leq 9C^*C_{\boldsymbol{\gamma},m}(\lambda) n^{-1/(2\lambda)},
	\end{equation}
	with
	\begin{equation}\label{eqn:qmcconst}
	C_{\boldsymbol{\gamma},m}(\lambda) =\left(\sum_{\emptyset\neq \boldsymbol{\mathfrak{u}}\subseteq \{1:m\}} \gamma_{\boldsymbol{\mathfrak{u}}}^{\lambda}\prod_{j\in\boldsymbol{\mathfrak{u}}}\varrho(\lambda)\right)^{1/(2\lambda)}\left(\sum_{\mathfrak{u}\subseteq \{1:m\}}\dfrac{(\lvert\boldsymbol{\mathfrak{u}}\rvert!)^2 T^{2\lvert\boldsymbol{\mathfrak{u}}\rvert}}{\gamma_{\boldsymbol{\mathfrak{u}}}\veps^{2\lvert\boldsymbol{\mathfrak{u}}\rvert}}
	\prod_{j\in\boldsymbol{\mathfrak{u}}} \lambda_j\lVert\phi_j\rVert_{C^0(\bar{D})}^2\right)^{1/2}.
	\end{equation}
\end{lemma}

\begin{proof} 
The proof of this result is essentially an application of the Koksma-Hlawka inequality, which is the same as the proofs of Theorem 15, Theorem 16, and Theorem 17 in \cite{graham2015quasi}, or Theorem 5.10 in \cite{dick2013high} with the following modification of estimates:
	\begin{equation}
	\varrho(\lambda) = 2\left(\dfrac{\sqrt{2\pi}}{\pi^{2-2\eta_*}(1-\eta_*)\eta_*}\right)^{\lambda}\zeta(\lambda+\frac12), \quad \eta_*=\frac{2\lambda-1}{4\lambda}
	\end{equation}
	with $\zeta(x)=\sum_{j=1}^{\infty} j^{-x}$ the Riemann zeta function,
	and $C^*=\lVert \mathcal{G} \rVert_{L^2(D)}$.
\end{proof}
To analyze the error of our method, we need to make some assumptions on the regularity of the 
eigenfunctions and the decay rate of the eigenvalues in the KL expansion \eqref{KLE-poterntial2} 
of the random potential.

\begin{assumption} \label{regularity-decay-kle}
\begin{enumerate}
\item[(a)] There exist $C>0$ and $\Theta>1$ such that $\lambda_j \leq C j^{-\Theta}$ for $j\ge 1$;
\item[(b)] The Karhunen-Lo\'{e}ve eigenfunctions $v_j(\bx)$ are continuous and there exist $C>0$ and $\eta\in [0,\frac{\Theta-1}{2\Theta})$ such that
		$\lVert v_j \rVert_{C^0(\bar{D})}\leq C\lambda_j^{-\eta}$ for $j\ge 1$;
\item [(c)] The sequence defined by $\sqrt{\lambda_j}\lVert v_j\rVert_{C^0(\bar{D})}, \; j\ge 1$ satisfies $\sum_{\substack{j\ge 1}} \left(\sqrt{\lambda_j}\lVert v_j\rVert_{C^0(\bar{D})}\right)^p < \infty$ for some $p\in (0,1]$, and $\sum_{\substack{j\ge 1}} \sqrt{\lambda_j}\lVert v_j\rVert_{C^0(\bar{D})} < \frac{\veps}{T}\sqrt{\varrho(\lambda)}$ for $\lambda\in (1/2,1]$.
\end{enumerate}
\end{assumption}	

Recall that $\psi^{\veps}$ and $\psi^{\veps}_{m}$ are solutions to \eqref{eqn:Sch} and \eqref{Schm}, respectively. 
Denote $\psi_{H,m}^{\veps}$ the solution obtained by our method using the multiscale reduced basis functions 
in the physical space and the qMC method in the random space. Under the assumptions for the random potential, we have the error estimate.
\begin{theorem}\label{main-result}
Consider the approximation of $\mathbb{E}[\mathcal{G}(\psi^{\veps})]$ via qMC multiscale finite element methods, denoted by $Q_{m,n}(\cdot;\mathcal{G}(\psi_{H,m}^{\veps}))$, where we assume $\psi^{\veps}\in L^2(\Omega; H^2(D))$. A randomly shifted lattice rule $Q_{m,n}$ is applied to $\mathcal{G}(\psi_m^{\veps})$. Then, we can bound the root-mean-square error with respect to the uniformly distributed shift $\boldsymbol{\Delta}\in[0,1]^m$ by
\begin{equation}
\sqrt{\mathbb{E}^{\boldsymbol{\Delta}} \left[\left(\mathbb{E}[\mathcal{G}(\psi^{\veps})]-Q_{m,n}(\cdot;\mathcal{G}(\psi_{H,m}^{\veps}))\right)^2\right]} \leq C\left(\frac{H^2}{\veps^2} + \frac{m^{-\chi}}{\veps} + n^{-r}\right),\quad 0< t\leq T,
\label{main-estimate-rate}
\end{equation}
for $0<\chi\leq(1/2-\eta)\Theta-1/2$, and with $r=1/p-1/2$ for $p\in(2/3,1]$ and $r=1-\delta$ for $p\leq 2/3$, with $\delta$ arbitrarily small. Here the constant $C$ is independent of $\veps$, $m$, and $n$ but depends on $T$.
\end{theorem}

\begin{proof}
The linearity of operator $\mathcal{G}$ implies
\begin{equation}\label{triangle-split}
\mathcal{G}(\psi^{\veps})-\mathcal{G}(\psi^{\veps}_{H,m}) =
\mathcal{G}(\psi^{\veps})-\mathcal{G}(\psi_H^{\veps})
+\mathcal{G}(\psi_H^{\veps})-\mathcal{G}(\psi_{H,m}^{\veps}).	
\end{equation}
Under the assumption $\psi^{\veps}\in L^2(\Omega; H^2(D))$, we have, see for example \cite{ChenMaZhang:prep},
\begin{equation}\label{OC-error}
	\lvert \mathbb{E}[\mathcal{G}(\psi^{\veps})-\mathcal{G}(\psi_H^{\veps})]
\rvert \leq C \frac{H^2}{\veps^2}.
\end{equation} 
Under the assumptions (b) and (c) in Assumption \ref{regularity-decay-kle}, we have, based on Lemma \ref{lemma:KLerror},
\begin{equation}\label{kle-truncation-error}
	\lvert \mathcal{G}(\psi_H^{\veps})-\mathcal{G}(\psi_{H,m}^{\veps}) \rvert
	\leq C \frac{m^{-\chi}}{\veps}
\end{equation} 
for all $0 < \chi \leq (1/2-\eta)\Theta-1/2$. Detailed derivation is essentially the same as the proof of Theorem 8 in \cite{graham2015quasi}.
	
Finally, when applying the qMC method to \eqref{triangle-split}, we need to analyze the error in the qMC method.
We adopt the standard framework, i.e., the Koksma-Hlawka inequality. Under \textbf{Assumption \ref{regularity-decay-kle}}, 
we have, based on Lemma \ref{lemma:qmc1} and Lemma \ref{lemma:qmc2},
\begin{equation}\label{qMC-error}
	\sqrt{\mathbb{E}^{\boldsymbol{\Delta}}\lvert I_m(F) - Q_{m,n}(\cdot; F) \rvert^2} \leq C n^{-r},
\end{equation} 
	where $r=1/p-1/2$ for $p\in(2/3,1]$
	and $r=1-\delta$ for $p\leq 2/3$, with $\delta$ arbitrarily small.
	Detailed derivation is essentially the same as the proof of Theorem 20 in \cite{graham2015quasi}.
	A combination of above estimates completes the proof.
\end{proof}
%\begin{remark}\label{explain-qMC1}
%The qMC method achieves a convergence rate of order $O(\frac{(\log n)^m}{n})$, where the result becomes pessimistic	when $m$ is moderate and large. However, we can improve this convergence by exploring the smoothness or structure in the solution $\psi_{m}^{\veps}$. In \eqref{qMC-error}, we get rid of the dimension dependence term $(\log n)^m$ in the numerator with a certain degree of deteriorate in the convergence rate, i.e., $r\in[1/2,1)$.  
%\end{remark}
\begin{remark}\label{explain-qMC22} 
The term $\frac{m^{-\chi}}{\veps}$ in the error estimate \eqref{main-estimate-rate} can be viewed as a modeling error. When the $m$-term KL truncation potential $v^{\epsilon}_{m}(\bx,\omega)$ in \eqref{KLE-poterntial2}
provides an accurate approximation to the potential $v^{\epsilon}(\bx,\omega)$, the term $\frac{m^{-\chi}}{\veps}$ becomes small or negligible. 
\end{remark}
\begin{remark}\label{explain-qMC2}
	In \secref{sec:NumericalExamples}, we will show the proposed method works well for a large class of random potentials, even when the eigenvalues 
	in the KL expansion have a relatively slow decay rate. Therefore, \textbf{Assumption \ref{regularity-decay-kle}} is a rather technical assumption
	for the convergence analysis of the proposed method.
\end{remark}
\begin{remark}\label{explain-qMC3} 
In the error analysis for the qMC method, we assume $\boldsymbol{\xi}=(\xi_1,...,\xi_m)\in [0,1]^m$ for notational convenience; see \eqref{eqn:integral_highdim}, where $\xi_i$ are i.i.d. uniform random variables. In the KL expansion \eqref{KLE-poterntial2} representation for $v_{m}^{\epsilon}(\bx,\omega)$, we choose $\xi_i\in [-\sqrt{3},\sqrt{3}]$, $i=1,...,m$ so that the conditions $\EEp{\xi_i}=0$, $\EEp{\xi_i\xi_j}=\delta_{ij}$ are satisfied. The same convergence result 
can be obtained with little modification of the current proof.
\end{remark}

\section{Numerical examples} \label{sec:NumericalExamples}
\noindent
In this section, we conduct numerical experiments to test the accuracy and the efficiency of our method. 
Specifically, we will present convergence tests with respect to the physical meshsize,  the number of multiscale reduced basis functions, and the number
of qMC samples. In addition, we will investigate the existence of Anderson localization in both 1D and 2D. For convenience, we first introduce $ L^2 $ norm and $ H^1 $ norm as
\[
||\psi^{\epsilon}||^2_{L^2}=\int_{D}|\psi^{\epsilon}|^2 \d\bx, \quad ||\psi^{\epsilon}||^2_{H^1}=\int_{D}|\nabla \psi^{\epsilon}|^2 \d\bx + \int_{D}|\psi^{\epsilon}|^2 \d\bx.
\]
In what follows, we compare the relative error between expectations of the numerical solution $\psi^{\epsilon}_{\textrm{num}}$ and the reference solution $\psi^{\epsilon}_{\textrm{ref}}$ in both $ L^2 $ norm and $ H^1 $ norm
\begin{align*}
\textrm{Error}_{L^2}& =\dfrac{||\mathds{E}[\psi^{\epsilon}_{\textrm{num}}]-\mathds{E}[\psi^{\epsilon}_{\textrm{ref}}]||_{L^2}}
{||\mathds{E}[\psi^{\epsilon}_{\textrm{ref}}]||_{L^2}},\\% \label{eqn:errorl2}\\
\textrm{Error}_{H^1}& =\dfrac{||\mathds{E}[\psi^{\epsilon}_{\textrm{num}}]-\mathds{E}[\psi^{\epsilon}_{\textrm{ref}}]||_{H^1}}
{||\mathds{E}[\psi^{\epsilon}_{\textrm{ref}}]||_{H^1}}.%\label{eqn:errorh1}.
\end{align*}
Here $\mathds{E}[\psi^{\epsilon}_{\textrm{num}}]=\int_{\Omega}\psi^{\epsilon}_{\textrm{num}}(t,\bx,\omega)\d \rho(\omega)$, $ \mathds{E}[\psi^{\epsilon}_{\textrm{ref}}]=\int_{\Omega}\psi^{\epsilon}_{\textrm{ref}}(t,\bx,\omega)\d\rho(\omega) $, $\Omega$ is the random space, and $\rho(\omega) $ is the probability measure induced by the randomness in \eqref{KLE-poterntial2}. 
The reference solution refers to the numerical wavefunction using a very fine mesh and a large amount of qMC samples.
In numerical experiments, we use the MATLAB's Statistics Toolbox to generate the Sobol sequence to implement the qMC method. 
When we use the POD method to construct multiscale reduced basis functions, we observed similar decay behaviors of the associated eigenvalues 
at each coarse grid point. Therefore, we choose the same reduced basis number $m_k$ for all the coarse grid points. 
 
\subsection{Convergence in the physical space} \label{sec:PhysicalSpace}
\noindent 
Consider the 1D Schr\"{o}dinger equation over $ D=[-\pi,\pi] $ 
\begin{align}
i\epsilon\partial_t\psi^\epsilon&=-\frac{\epsilon^2}{2}\partial_{xx}\psi^\epsilon+v^{\epsilon}(x,\omega) \psi^\epsilon,
\label{SchrodingEq-1d}
\end{align}
where the periodic condition is imposed, the initial data $ \psi_{\textrm{in}}(x)=(\frac{10}{\pi})^{1/4}e^{-20(x-0)^2} $, and the random potential  $ v^{\epsilon}(x,\omega) $ is defined as
\begin{align}
v^{\epsilon}(x,\omega)=1+\sigma \sum_{j=1}^{3} \sin(jx^2)\sin(\frac{x}{E_j})\xi_j(\omega). 
\label{random-potential-1d}
\end{align}
In the random potential \eqref{random-potential-1d}, $\sigma$ is used to control the strength of the random potential, and $\xi_j(\omega)$'s are mean-zero and independent random variables uniformly distributed in $ [-\sqrt{3},\sqrt{3}] $. Moreover, we choose $ \veps=\frac{1}{16} $, $\sigma=1$ and $ E=[\frac{1}{9},\frac{1}{13},\frac{1}{11}] $, i.e., the characteristic length scale of randomness is
different from the semiclassical parameter.

\textbf{Convergence with respect to the coarse mesh size $ H $.}  In our numerical test, we set the final computational time $ T=1 $. For the reference solution, we choose the fine mesh to be $ h=\frac{2\pi}{2048} $ and the qMC sample number to be $ n=16000 $. For our method, we choose the POD modes $ m_k=3$, the sampling number in the offline training stage to be $200$ and the number of qMC samples in the online stage to be $2560$. 

In Table \ref{table_physical}, we compute the relative errors of the expectation of the wavefunction in
both $L^2$ norm and $H^1$ norm for a series of coarse meshes with meshsize ranging from $ H=\frac{2\pi}{32} $ to $ H=\frac{2\pi}{256} $. 
Nice convergence in the physical space is observed.
	\begin{table}[H]
	\centering
	\begin{tabular}{|r|r|r|r|r|}
		\hline
		$ H $ & $ \textrm{Error}_{L^2} $ & Order & $ \textrm{Error}_{H^1} $ & Order \\
		\hline
		$ 2\pi/32  $ &0.09862312&       &0.32096262&     \\
		$ 2\pi/64  $ &0.00129644& 6.25  &0.01449534& 4.47\\
		$ 2\pi/128 $ &0.00002892& 5.49  &0.00076150& 4.25\\
		$ 2\pi/256 $ &0.00000950& 1.61  &0.00014161& 2.42\\
		\hline
	\end{tabular}%
	\caption{Relative $L^2$ and $H^1$ errors for the expectation of the wavefunction when $ \epsilon=1/16 $.}
	\label{table_physical}%
\end{table}
 
\textbf{Verification of the exponential decay of multiscale basis functions.} For the same problem as above, we choose four different realizations of the multiscale basis functions centered at $x=0$, i.e. $ \phi(x,\boldsymbol{\xi}(\omega_i)) $, $i=1,2,3,4$, which are generated in the offline training stage of our previous experiment when $ H=\frac{2\pi}{256} $. In Figure \ref{fig:gradienta}, we plot $|\nabla\phi(x,\boldsymbol{\xi}(\omega_i))| / ||\nabla\phi(x,\boldsymbol{\xi}(\omega_i))||_{L_2(D)}$, $i=1,2,3,4$. In Figure \ref{fig:exponentialb}, we plot the quantity $ E_{\textrm{relative}}=\frac{||\nabla \phi(x,\boldsymbol{\xi}(\omega_i)) ||_{L_2(D)}-||\nabla \phi(x,\boldsymbol{\xi}(\omega_i)) ||_{L_2(D_{\ell})}}{\max(||\nabla \phi(x,\boldsymbol{\xi}(\omega_i)) ||_{L_2(D)}-||\nabla \phi(x,\boldsymbol{\xi}(\omega_i)) ||_{L_2(D_{\ell})})}$ with respect to the patch size $\ell$, which shows the decay rate of $E_{\textrm{relative}}$ with respect to $\ell$. 

One can see that each realization of the multiscale basis functions decays exponentially fast away from the center $x=0$. Since the multiscale basis functions have exponential decay property, the approximated multiscale basis using the reduced basis functions (see \eqref{RB_expansion}) still has the same property.  
 
 \begin{figure}[htbp]
		\centering
		\begin{subfigure}{0.49\textwidth}
			\centering
			\includegraphics[width=\textwidth]{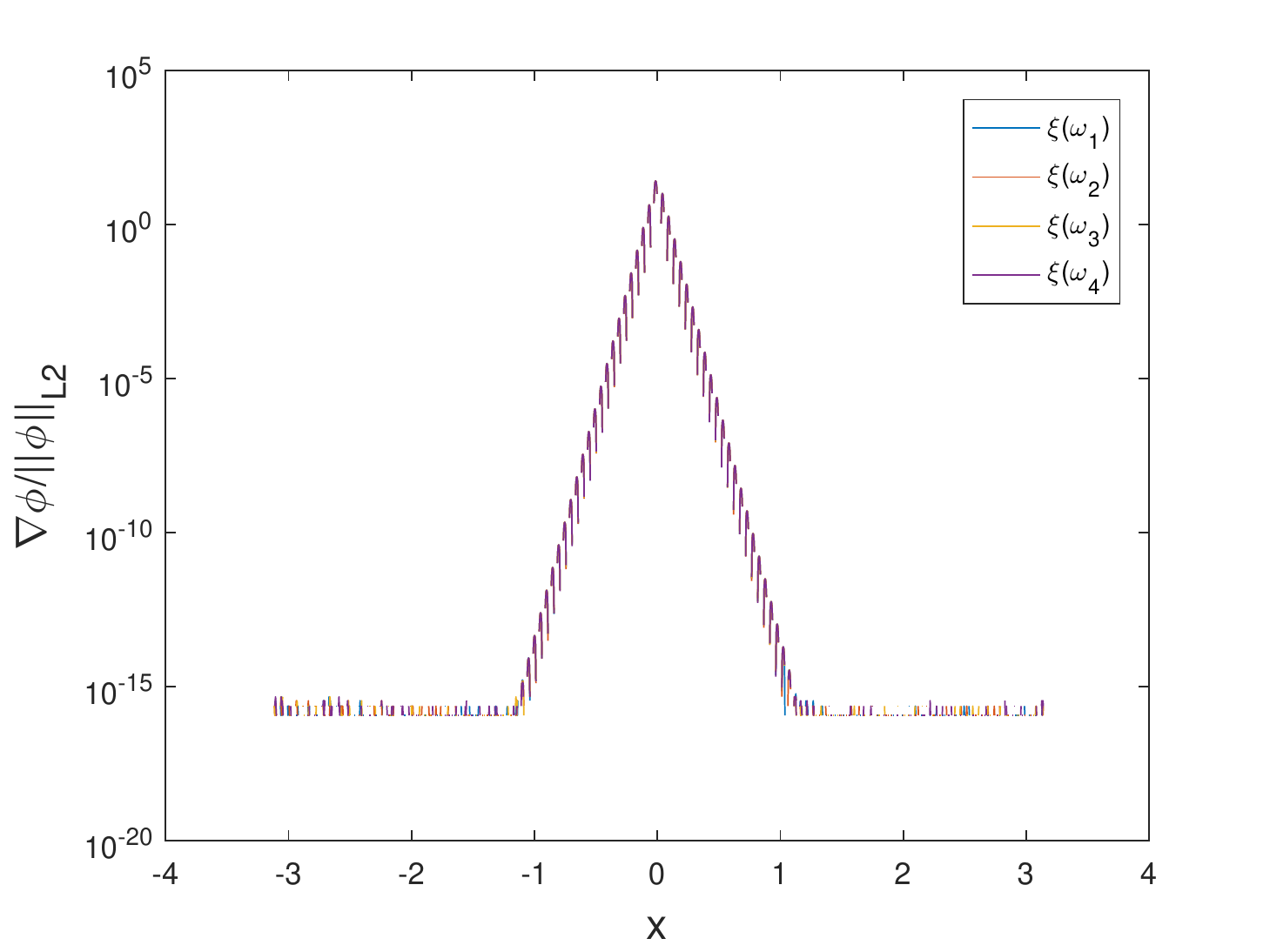}
			\caption{ {\scriptsize $\nabla\phi / ||\phi||_{L_2}$ with respect to the distance to $x=0$}}
			\label{fig:gradienta}
		\end{subfigure}
		\begin{subfigure}{0.49\textwidth}
			\centering
			\includegraphics[width=\textwidth]{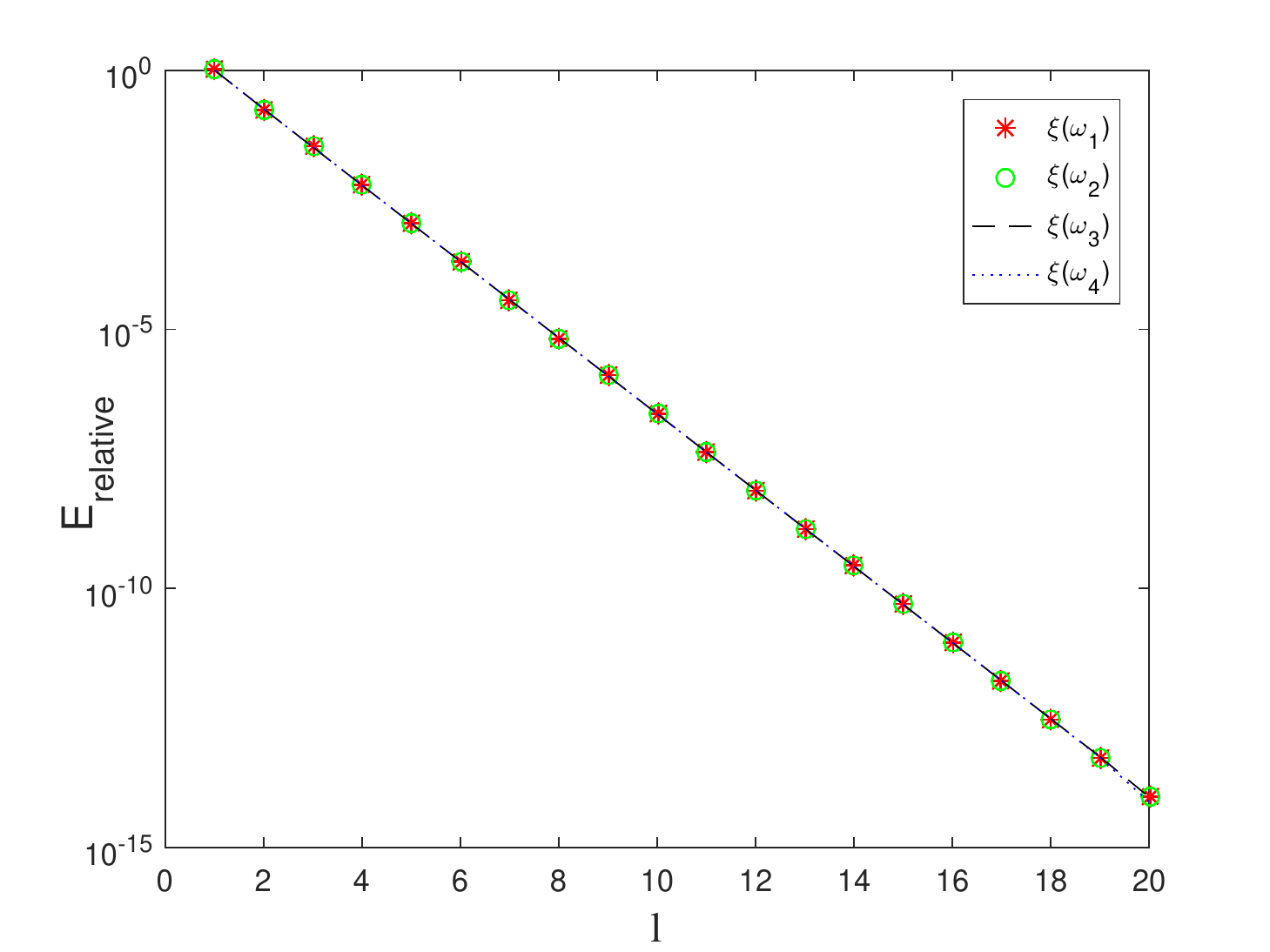}
		    \caption{{\scriptsize $ E_{\textrm{relative}}=\frac{||\nabla \phi(x,\xi(\omega_i)) ||_{L_2(D)}-||\nabla \phi(x,\xi(\omega_i)) ||_{L_2(D_{\ell})}}{\max(||\nabla \phi(x,\xi(\omega_i)) ||_{L_2(D)}-||\nabla \phi(x,\xi(\omega_i)) ||_{L_2(D_{\ell})})}$.}}
			\label{fig:exponentialb}
		\end{subfigure}
		\caption{Exponentially decaying properties of the multiscale basis functions for four different realizations. %Left: $\nabla\phi / ||\phi||_{L_2}$ with respect to the distance to $x=0$; Right: $ E_{\textrm{relative}}=\dfrac{||\nabla \phi ||_{D}-||\nabla \phi ||_{D_{\ell}}}{\max(||\nabla \phi ||_{D}-||\nabla \phi ||_{D_{\ell}})}$ with respect to the patch size $\ell$.
		}
		\label{exponential_decay}
	\end{figure} 

\textbf{Convergence with respect to the number of multiscale reduced basis functions.} We study how the approximation error depends on the number of multiscale reduced basis used at each coarse mesh node $\bx_k$, i.e., changing the POD modes $ m_k $. Again, we solve \eqref{SchrodingEq-1d} - \eqref{random-potential-1d} when $ \veps=\frac{1}{16} $, $\sigma=1$ and $ E=[\frac{1}{9},\frac{1}{13},\frac{1}{11}] $. The final computational time $ T=1 $. 
For the reference solution, we choose the meshsize to be $ h=\frac{2\pi}{2048} $ and the number of qMC samples to be $ n=16000 $. For our method, we choose the  number of samples in the offline training stage to be $200$ and the number of qMC samples in the online stage to be $2560$. We fix the coarse mesh size $ H=\frac{2\pi}{128} $ and record the relative errors as a function of the number of multiscale reduced basis functions.

In Figure \ref{fig:POD}, we plot the relative $L^2$ and $H^1$ errors with respect to the number of multiscale reduced basis functions. It is observed that  
results when $ m_k=2 $ or $ m_k=3 $ have already been good enough in the sense that relative errors are less than $1\%$. 
These numerical results indicate that multiscale reduced basis functions can efficiently approximate the physical space of the wavefunction. 
 
      \begin{figure}[htbp]
		\centering
		\begin{subfigure}{0.49\textwidth}
			\centering
			\includegraphics[width=\textwidth]{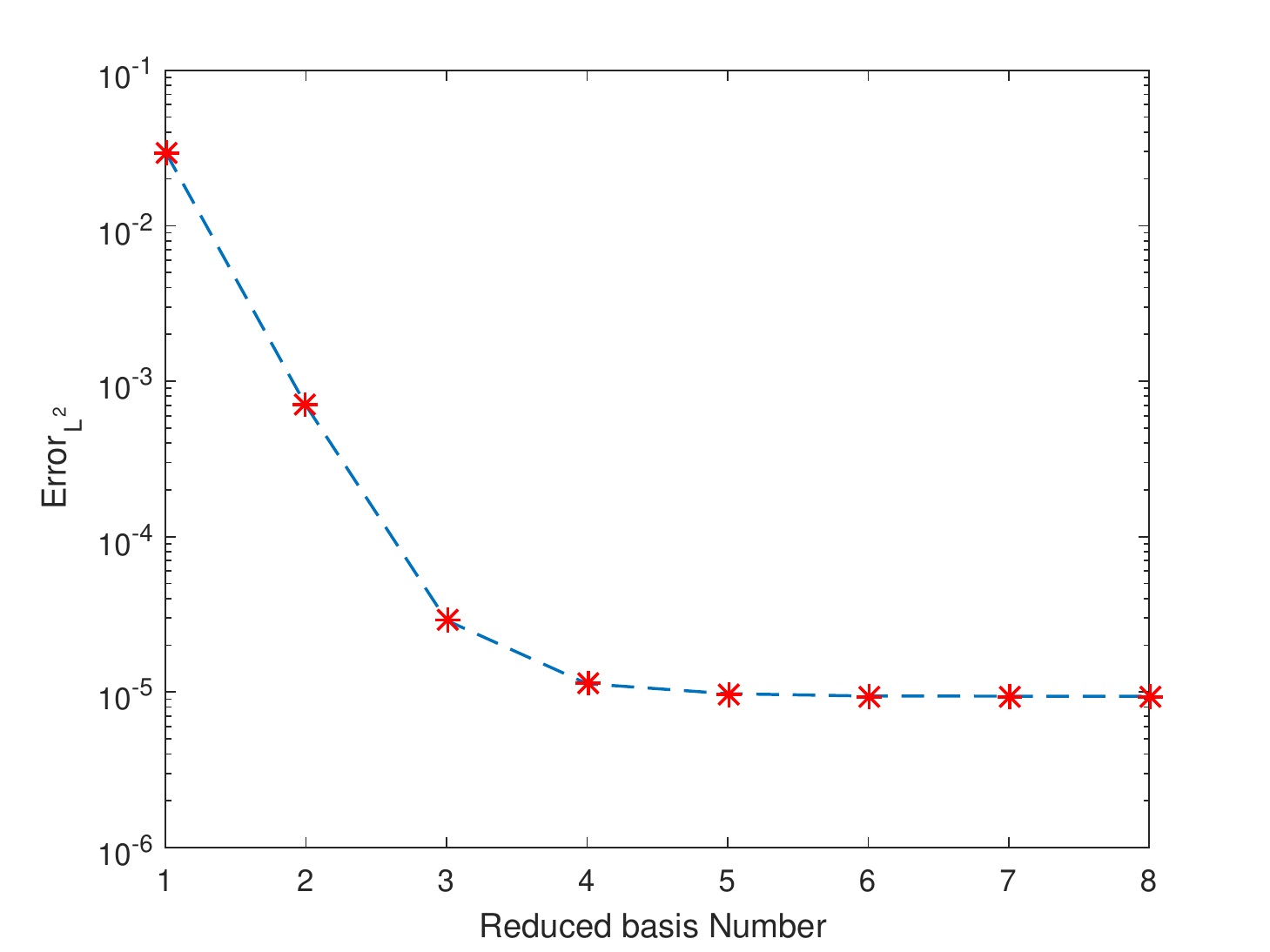}
			\caption{Relative error in $ L^2 $ norm.}
			\label{fig:POD_L2}
		\end{subfigure}
		\begin{subfigure}{0.49\textwidth}
			\centering
			\includegraphics[width=\textwidth]{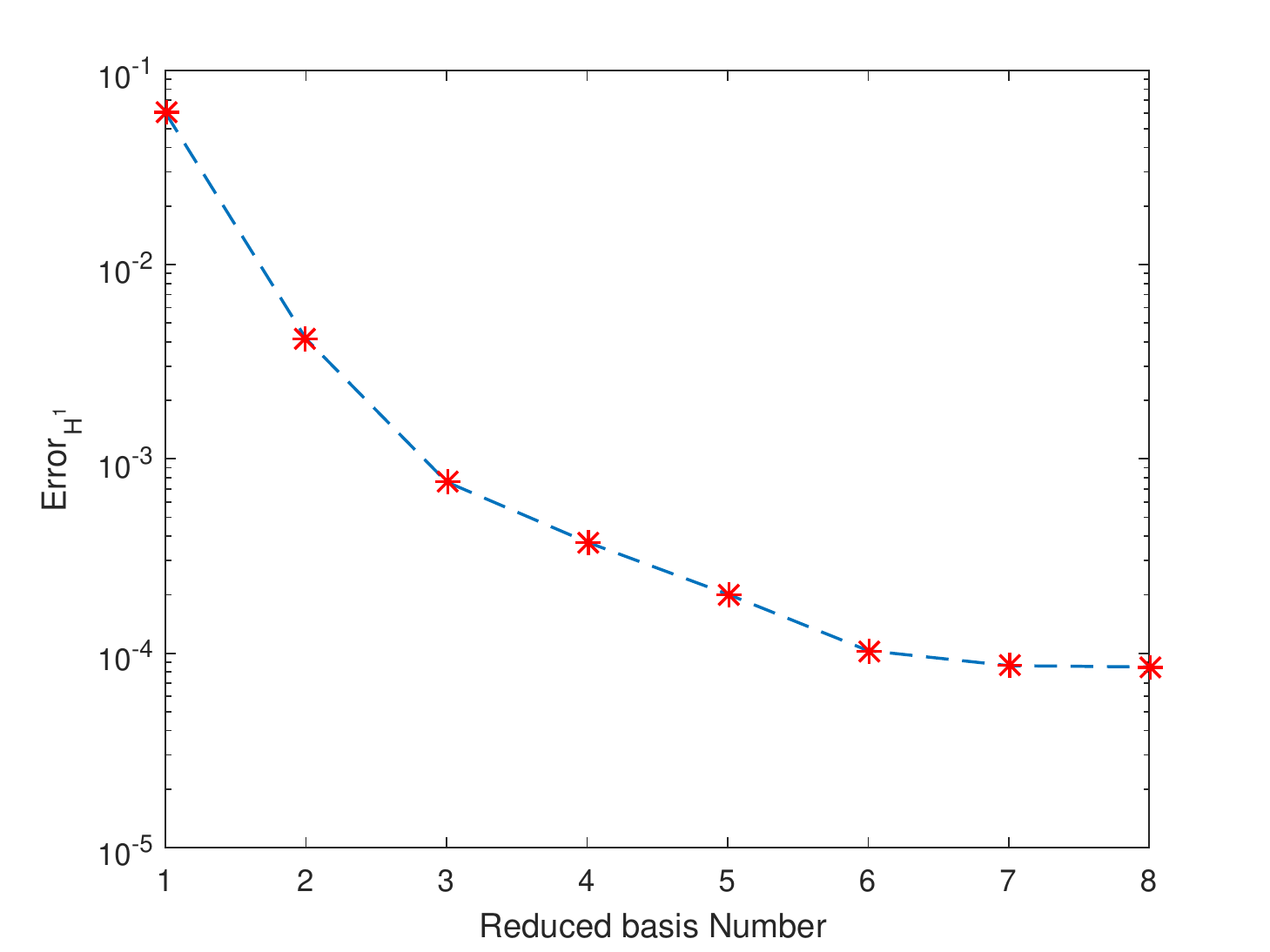}
			\caption{Relative error in $ H^1 $ norm.}
			\label{fig:POD_H1}
		\end{subfigure}
		\caption{Relative errors with respect to the number of the multiscale reduced basis functions. }
		\label{fig:POD}
	\end{figure}

\subsection{Convergence in the random space} \label{sec:RandomSpace}
\noindent
Again, we use the same example: \eqref{SchrodingEq-1d} - \eqref{random-potential-1d} and $ D=[-\pi,\pi] $, but we shall focus on the convergence 
of our method in random space. 
 
\textbf{Convergence with respect to the number of qMC samples.}  In this numerical experiment, parameters of the random potential are the same as those in
\secref{sec:PhysicalSpace}, i.e., $\sigma=1$ and $ E=[\frac{1}{9},\frac{1}{13},\frac{1}{11}] $. Set $ \veps=\frac{1}{16} $ and the final time $T=1$. For the reference solution, we choose the meshsize to be $ h=\frac{2\pi}{2048} $ and the number of qMC samples to be $ n=16000 $. For our method, we choose the coarse meshsize to be $ H=\frac{2\pi}{256} $ and the number of multiscale reduced basis functions to be $ m_k=4$, such that the error in the physical space be small enough. To study the convergence rate of the qMC method, we change the number of the qMC samples successively from $ n=160 $ to $ n=5120 $ and compute the relative $L^2$ errors. We also compute the relative errors of the MC method with the same setting in the physical space and the same number of samples.

In Figure \ref{fig:QMC}, we show the convergence result of our method. We find that the convergence rate of the qMC method is close to $O(n^{-1})$, which is consistent with results in \textbf{Lemma \ref{lemma:qmc2}} and in \textbf{Theorem \ref{main-result}}. Meanwhile, we compare the performance of the qMC method and the MC method. One can see that the convergence rate of the MC method is close to $O(n^{-\frac{1}{2}})$, which is also consistent with the error estimate of the MC method. This result clearly show that qMC method is more accurate and efficient than the MC method.  
%	\begin{figure}[htbp]
%		\centering
%		\begin{subfigure}{0.49\textwidth}
%			\centering
%			\includegraphics[width=\textwidth]{figs/EX_QMC_L2.pdf}
%			\caption{Convergence rate of the qMC method. 
%				The slope of the fitted line is $1.13$.}
%			\label{fig:QMC_L2}
%		\end{subfigure}
%		\begin{subfigure}{0.49\textwidth}
%			\centering
%			\includegraphics[width=\textwidth]{figs/EX_QMC_MC.pdf}
%			\caption{Comparison of the qMC method and the MC method. Convergence rate are $1.13$ and $0.57$, respectively.}
%			\label{fig:QMC_MC}
%		\end{subfigure}
%		\caption{Convergences for the qMC method and the MC method.}
%		\label{fig:QMC}
%	\end{figure} 
\begin{figure}[htbp]
	\centering
	\includegraphics[width=0.50\linewidth]{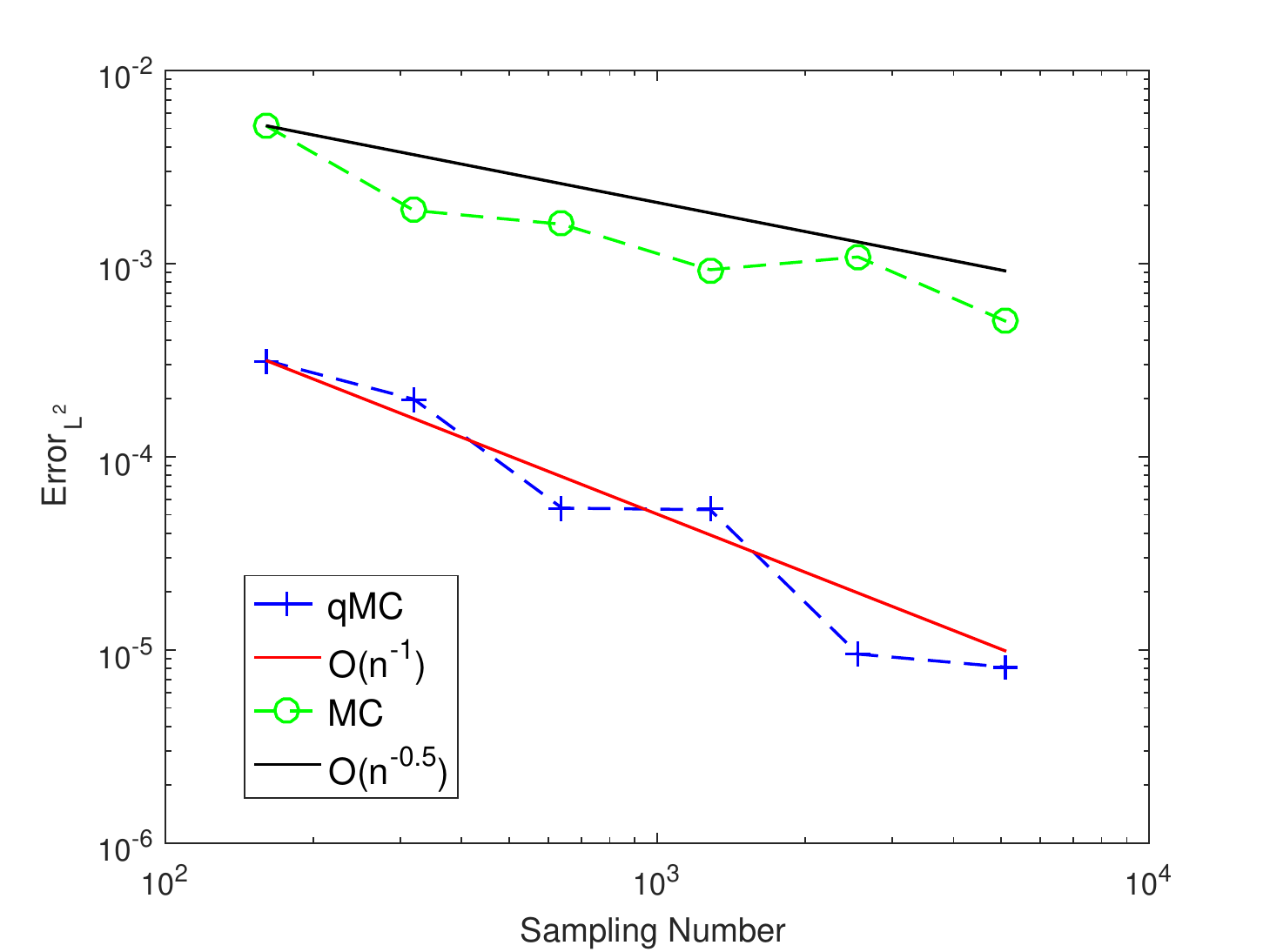}
	\caption{Comparison of the qMC method and the MC method. Convergence rate for qMC and MC are $1.13$ and $0.57$, respectively.}
	\label{fig:QMC}
\end{figure}

\textbf{Estimation of sampling numbers in the construction of multiscale reduced basis functions.} 
In \secref{sec:DetermineSamples}, we obtain qualitative estimates on the choice of sampling numbers in the construction of multiscale reduced basis functions; see \eqref{asd} and \eqref{basis-potential}. In this experiment, we first generate $Q$ qMC samples of the random potential: $\{v^{\epsilon}(x,\omega_q)\}_{q=1}^{Q}$. Then, for each sample $v^{\epsilon}(x,\omega_q)$, we compute the corresponding multiscale basis functions. Finally, we construct multiscale reduced basis functions using the POD method. In the online stage, we solve \eqref{Schm} using the obtained multiscale reduced basis functions. The numerical setting for the reference solution is the same as before. For our method, we choose $ H=\frac{2\pi}{128} $, $ m_k=3$, and $n=2560$.

In Table \ref{table_learning}, we show relative errors of numerical solutions obtained using different sampling numbers of the random potential. When the sampling number $Q$ is small, say $Q=10$, the error is big and the corresponding multiscale reduced basis functions cannot approximate the random space of the wavefunction well. When we increase $Q$, i.e., add more samples of the random potential in the construction of multiscale reduced basis functions, we obtain much better results. Notice that $m_k$ is fixed to be $3$. This means when $Q$ is of order $100$, the sampling number of the random potential is large enough to ensure the excellent approximation accuracy of multiscale reduced basis functions. One interesting topic on this issue is an optimal sampling strategy in the construction of multiscale reduced basis functions, which will be explored in a subsequent work.
	
\begin{table}[htbp]
	\centering
	\begin{tabular}{|r|r|r|}
		\hline
		qMC number & $ \textrm{Error}_{L^2} $ &$ \textrm{Error}_{H^1} $\\
		\hline
		$ 10  $ &0.11800774 & 0.46614288\\
		$ 100 $ &0.00136249 & 0.01497658 \\
		$ 200 $ &0.00130909 & 0.01455442 \\
		$ 400 $ &0.00129678 & 0.01449570 \\
		\hline
	\end{tabular}%
	\caption{Relative $L^2$ and $H^1$ errors in terms of sampling numbers of the qMC method in the offline stage.}
	\label{table_learning}%
\end{table}

\textbf{Dependence of the number of qMC samples on $ \epsilon $ and dimension of the random space $ m $.} 
We use the random potential $ v^{\epsilon}(x,\omega) $ with decaying terms satisfying \textbf{Assumption} \ref{regularity-decay-kle}:
\begin{align}
	v^{\epsilon}(x,\omega)=1+\sum_{j=1}^{m}\frac{1}{j^2}\sin(jx)\xi_j(\omega)
\end{align}
in 1D physical domain $ D=[-\pi,\pi] $ and $\xi_j(\omega)$'s are mean-zero and independent random variables uniformly distributed in $ [-\sqrt{3},\sqrt{3}] $. 

Firstly, we set random dimension to be $ m=8 $, the final time $ T=1$. Three values of $\epsilon=\frac{1}{4}, \frac{1}{8}$ and $\frac{1}{16}$ are tested. The reference solution is obtained in the same way as before. For the numerical solution we use the same fine mesh as that for the reference solution but different number of qMC samples. In Table \ref{table_qMC1_beta2}, we list the number of qMC samples with respect to $\epsilon$
for the same accuracy requirement. It is observed that the number of qMC samples increases proportionally to $1/\veps^{2.5}$.
\begin{table}[H]
	\centering
	\begin{tabular}{|r|r|r|r|}
		\hline
		$\epsilon$ & qMC number & $ \textrm{Error}_{L^2} $ &$ \textrm{Error}_{H^1} $\\
		\hline
		$ 1/4  $ &160  & 0.00469003&0.00654782\\
		$ 1/8  $ &960  & 0.00399369&0.00767395\\
		$ 1/16 $ &5120 & 0.00444144&0.00785192\\
		\hline
	\end{tabular}%
	\caption{{\small Number of qMC samples for different $\epsilon$ under the same accuracy requirement.}}
	\label{table_qMC1_beta2}%
\end{table}

Secondly, we fix $\epsilon=\frac{1}{16}$ and change the dimension of the random space from $ m=1 $, $ m=2 $, $ m=4 $, to $ m=8 $. The reference solution and numerical solution are obtained in the same way as above. In Table \ref{table_qMC2_beta2}, we list the number of qMC samples with respect to $m$ for the same accuracy requirement. A linear growth of the number of qMC samples is observed when $m$ is increased.
 
\begin{table}[H]
	\centering
	\begin{tabular}{|r|r|r|r|}
		\hline
		Dimension $ m $ & qMC number & $ \textrm{Error}_{L^2} $ &$ \textrm{Error}_{H^1} $\\
		\hline
		$ 1  $ &520  & 0.00405535&0.00784524\\
		$ 2  $ &1280 & 0.00341203&0.00667093\\
		$ 4  $ &2560 & 0.00369911&0.00823515\\
		$ 8  $ &5120 & 0.00444144&0.00785192\\
		\hline
	\end{tabular}%
	\caption{{\small Number of qMC samples for different dimension $m$ under the same accuracy requirement.}}
	\label{table_qMC2_beta2}%
\end{table}

A slower decay of eigenvalues in the KL expansion of the random potential requires more qMC samples. For instance,  when $v^{\epsilon}(x,\omega)=1+\sum_{j=1}^{m}\frac{1}{j}\sin(jx)\xi_j(\omega)$, we observed a quadratic growth of the number of qMC samples 
when $m$ is increased. However, the qMC method is still very efficient in solving this difficult problem. 
Moreover, the qMC method can be implemented in a parallel fashion to further improve its efficiency.
 
\subsection{Investigation of Anderson localization.} \label{sec:AndersonLocalization}
\noindent 
In this section, we investigate the Anderson localization phenomenon for the semiclassical Schr\"{o}dinger equaiton using our method. 
Physically, when the Anderson localization happens, the electron transport stops under the strong disorder and the short-range correlation in space.
We emphasize that the short-range correlation is important for localization, while the long-range correlation may lead to delocalization \cite{Erdos:10,nosov2019correlation}. To numerically measure the localization of a wavefunction, we define 
\begin{align} 
A(t)=\mathds{E}\left[\int_{D}|\bx|^2|\psi^\epsilon(t,\bx,\omega)|^2\mathrm{d}\bx\right],
\label{second_moment}
\end{align}
where $\bx=x$ when $d=1$ and $\bx=(x_1,x_2)$ when $d=2$.

%	In order to show Anderson Localization, we test the 1D whitenoise potential because this potential will be random and very short-range correlated. The potential $ v^{\epsilon}(x,\omega) $ is defined by:
%	\begin{align}
%	v^{\epsilon}(x,\omega)=\sigma U(x)
%	\end{align}
%	In which, U(x) means that for each spatial node $ x_i $ is under uniform distribution in $ [-\sqrt{3},\sqrt{3}] $  and they are mutually independent. The physical domain $ D=[-\pi,\pi] $. And the fine scale mesh size is chose to be $ h=\frac{2\pi}{400} $  and the coarse mesh size is $ H=\frac{2\pi}{80} $. The $ \sigma $ in potential is used to control the strength of randomness. The Planck constant in this case is $ \veps=\frac{1}{4} $ and the initial date $ \psi_{\textrm{in}}(x) $ is 	
%	\begin{equation}\label{eqn:ic1d}
%	\psi_{\textrm{in}}(x)=(\frac{10}{\pi})^{1/4}e^{-20(x-0)^2},
%	\end{equation}
%	Below we will give the figure to demonstrate the temporal behavior of $ A(t) $ computed by our method.
%	\begin{figure}[H]
%		\centering
%		\includegraphics[width=0.50\linewidth]{figs/EX_whitenoise.pdf}
%		\caption{Phenomenon of Anderson localization for different $\sigma$}
%		\label{fig:whitenoise}
%	\end{figure}
%	Observed from the Figure \ref{fig:whitenoise}, we can find when there is random in potential, $ A(t) $ tends to be converged and remains unchanged with time $ t $. Which means that the diffusion process of electron stops.		
	
\textbf{1D Schr\"{o}dinger equation.} Consider the Schr\"{o}dinger equation \eqref{SchrodingEq-1d} with the periodic boundary condition over $ D=[-\pi,\pi] $. 
To approximate the spatially white noise in the potential, we employ the $m$-term KL expansion
\begin{align}
   v^{\epsilon}(x,\omega)=\sigma\sum_{j=1}^{m} \sin(jx) \frac{1}{j^{\beta}}\xi_j(\omega),
   \label{KLE-potential-1D-AL}
\end{align}
where $\xi_j(\omega)$'s are mean-zero and i.i.d. random variables uniformly distributed in $ [-\sqrt{3},\sqrt{3}] $. When $\beta=0$, 
$v^{\epsilon}(x,\omega) $ converges to the spatially white noise as $ m \rightarrow\infty$.  $ \sigma $ controls the strength of randomness. 

The setup is as follows: the fine scale meshsize $ h=\frac{2\pi}{600} $, the coarse meshsize $ H=\frac{2\pi}{100} $, $ \veps=\frac{1}{8} $, 
$\sigma=5$, and the initial data $ \psi_{\textrm{in}}(x) $ is 	
\begin{align}\label{eqn:ic1d}
	\psi_{\textrm{in}}(x)=(\frac{10}{\pi})^{1/4}e^{-20(x-0)^2}.
\end{align}
In Figure \ref{fig:beta0}, we plot $ A(t) $ as a function of $t$ for different $ m $ when $\beta=0 $. When $ m $ increases, 
the wavefunction quickly enters a localization phase. In Figure \ref{fig:beta1}, 
we plot the time evolution of $ A(t) $ for different $ m $ when  $\beta=1 $. 
Notice that $\beta=1 $ leads to a slower decay in the KL expansion of the random potential \eqref{KLE-potential-1D-AL}. 
Therefore, more terms need to be added to the KL expansion in order to generate a localization phase for the wavefunction.
We also plot the time evolution of $A(t)$ for $\beta$ ranging from $ 0 $ to $ 1.5 $ when $\sigma=5$, $ m=15 $ in Figure \ref{fig:N15}.
The localization phase is much easier to be approached as $\beta$ goes to $0$. Besides, we also observe that a larger $\sigma$
makes the wavefunction approach the localization phase more quickly with other parameters fixed.
To sum up, the localization phase can be approached easier when we have more terms in the KL expansion,  shorter range of randomness,
or stronger randomness.

\begin{figure}[htbp]
		\centering
		\begin{subfigure}{0.49\textwidth}
			\centering
			\includegraphics[width=\textwidth]{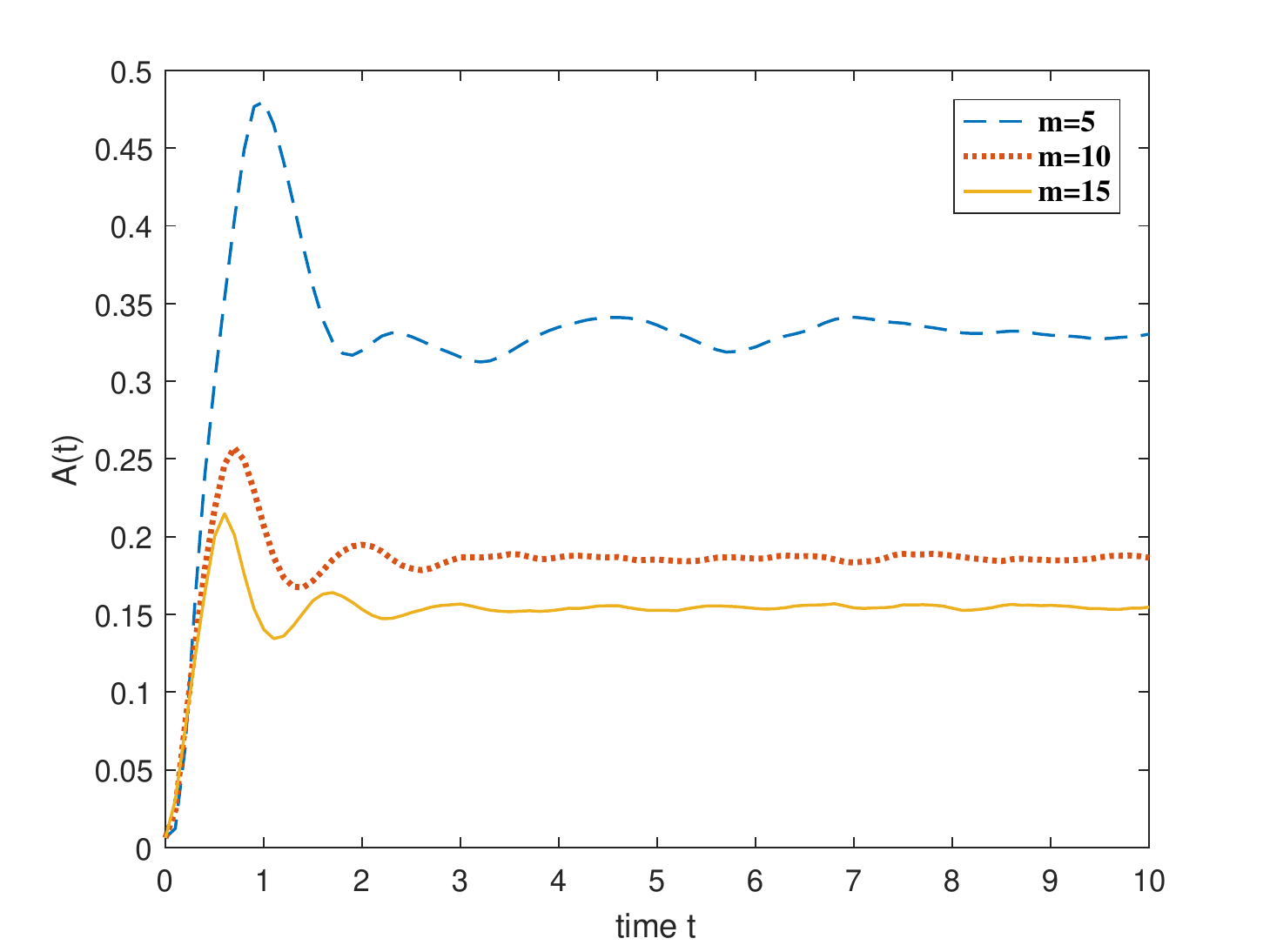}
		    \caption{$\beta=0$}
			\label{fig:beta0}
		\end{subfigure}
		\begin{subfigure}{0.49\textwidth}
			\centering
			\includegraphics[width=\textwidth]{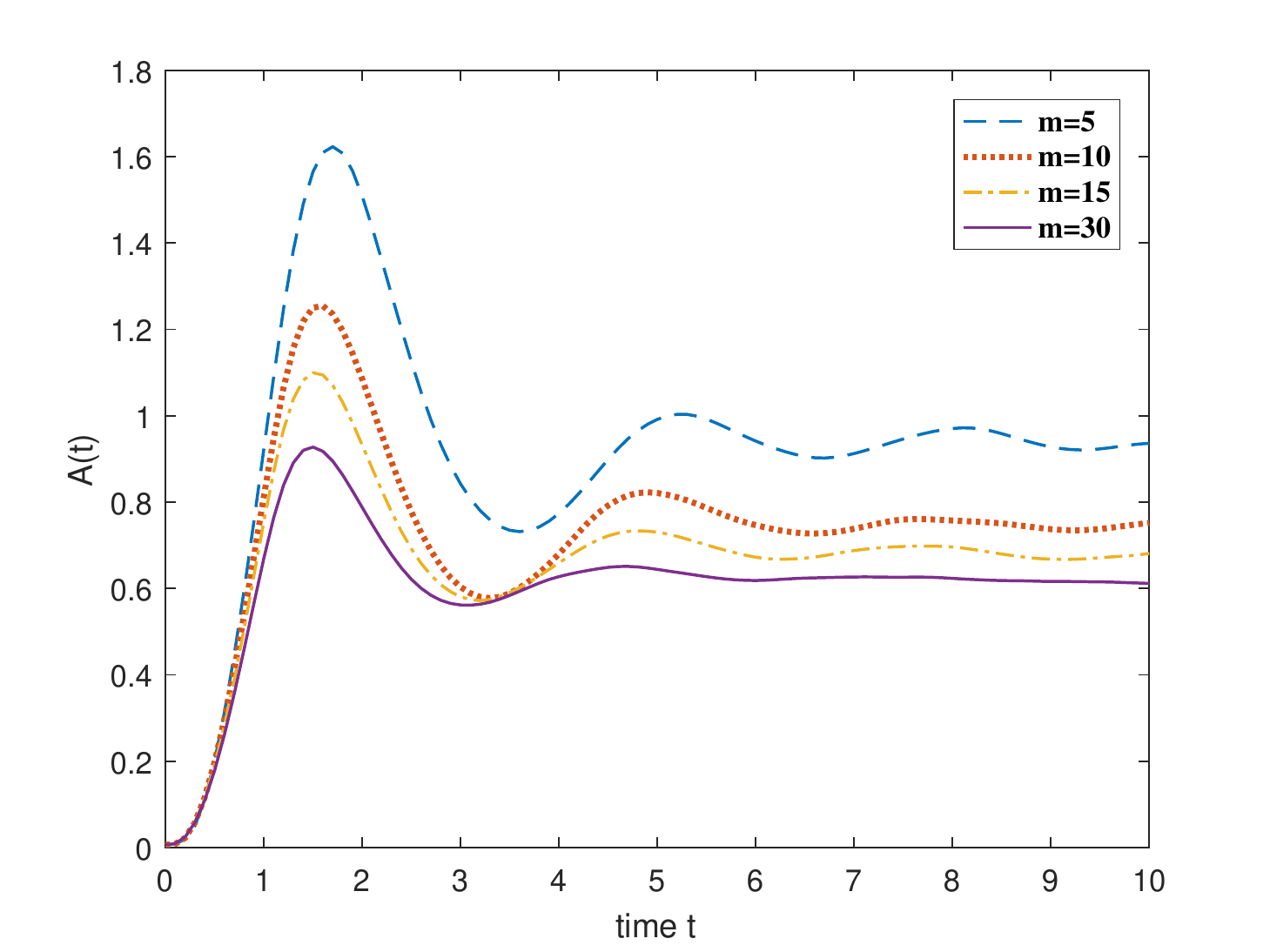}
			\caption{$\beta=1$}
			\label{fig:beta1}
		\end{subfigure}
		\caption{Anderson localization for different parameters. 
		%	Left: $\beta=0$; Right: $\beta=1$.
		}
		\label{fig:kle}
\end{figure}

\begin{figure}[htbp]
		\centering
		\includegraphics[width=0.50\linewidth]{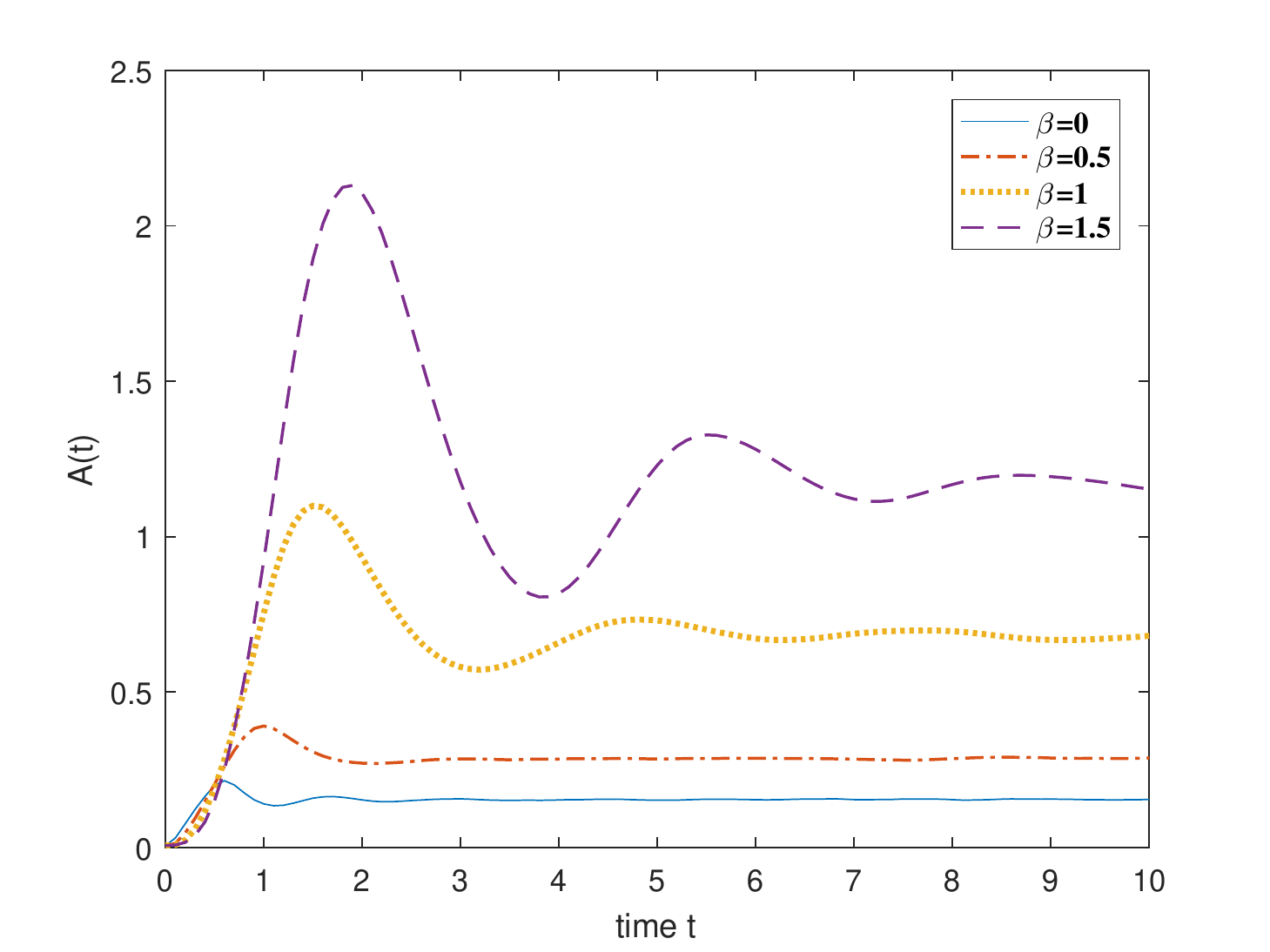}
		\caption{Anderson localization for different $\beta$.}
		\label{fig:N15}
\end{figure}
%In Figure \ref{fig:N15}, smaller $\beta$, more obvious the Anderson localization. While for big $\beta$ which means the random field is far from short-ranged random field, therefore the A(t) will be more fluctuated with respect to time compared with $\beta=0$.

\textbf{2D Schr\"{o}dinger equation.} Consider the Schr\"{o}dinger equation \eqref{Schm} over $ D=[-\pi,\pi]\times[-\pi,\pi] $ and
\begin{align}
	v^{\epsilon}(x_1,x_2,\omega)=\sigma\sum_{j=1}^{m} \sin(jx_1) \sin(jx_2)\frac{1}{j^{\beta}}\xi_j(\omega),
	\label{KLE-potential-2D-AL}
\end{align}
where the setting of $\xi_j(\omega)$'s is the same as the 1D case. $\sigma$, $ m $ and $\beta$ are parameters that controls the random potential. 

Choose $\sigma=5$, $\beta=0$ and $ \veps=\frac{1}{4} $. Notice that $\beta=0$ and \eqref{KLE-potential-2D-AL} is used to model a short-range random potential. For our method, the fine meshsize is $ h=\frac{2\pi}{400} $ and the coarse meshsize is $ H=\frac{2\pi}{100} $. In Figure \ref{fig:le2d}, we plot the time evolution of $ A(t) $ when $ m=10 $. One can see that the wavefunction approaches a localization phase when $t=4$. We remark that it is computationally expensive to solve the 2D Schr\"{o}dinger equation with random potentials. The proposed method, however, is efficient to solve this problem.
 
\begin{figure}[htbp]
		\centering
		\includegraphics[width=0.50\linewidth]{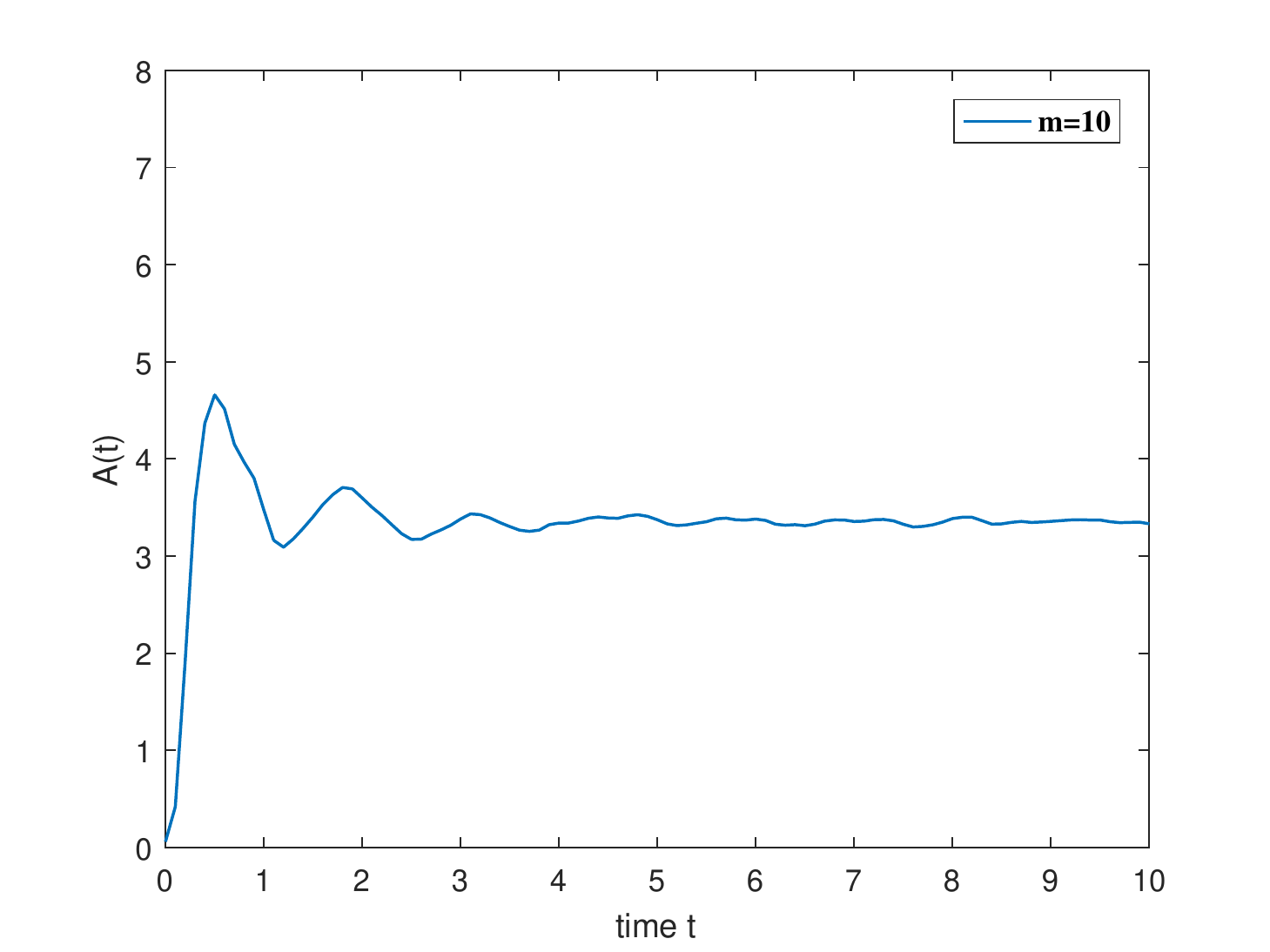}
		\caption{Anderson localization when $\sigma=5$ and $\beta=0$ in 2D.}
		\label{fig:le2d}
\end{figure}

\section{Conclusions and discussions} \label{sec:Conclusion}
\noindent
In this paper, we propose a multiscale reduced basis method to solve the Schr\"{o}dinger equation with random potential in the semiclassical regime.
The physical space of the solution is approximated by a set of localized multiscale basis functions based on an optimization approach. The proper orthogonal
decomposition method is then applied to extract a smaller number of multiscale reduced basis functions to further reduce the computational cost without loss
of approximation accuracy. The number of samples to learn the multiscale reduced basis functions is also analyzed, which provides guidance in practical
computations. The quasi-Monte Carlo method is employed to approximate the random space of the solution. Approximation accuracy of the proposed method
is analyzed. It is observed that the spatial gridsize is proportional to the semiclassical parameter and the number of samples is inversely proportional to the same parameter.  Finally we present several numerical examples to demonstrate the accuracy and efficiency of the proposed method. 
Moreover, we investigate the Anderson localization phenomena for Schr\"{o}dinger equation with correlated random potentials in both 1D and 2D.

There are two lines of work which deserve explorations in the near future.
Firstly, in the physics community, the random Schr\"{o}dinger equation in higher dimensions (2D and 3D) has been frequently used to study Anderson localization;
see \cite{filoche2012universal} for example.  Though the random potential is assumed to be white noise without 
spatial correlation in the original paper \cite{Anderson:58}, correlated random potential is also found to generate localized states; see \cite{devakul2017anderson}
for example. In the mathematics community, it is also known that the existence or nonexistence of Anderson localization for some types of 3D Schr\"{o}dinger equations with random potentials remains open \cite{Erdos:10}. It is thus quite interesting to explore this issue from a numerical perspective. 
Secondly, we plan to solve the Helmholtz equation in random media using the multiscale reduced basis basis method developed in this paper.

\section*{Acknowledgements}
\noindent
J. Chen acknowledges the financial support by National Natural Science Foundation of China via grant 21602149.	Z. Zhang acknowledges the financial support of Hong Kong RGC grants (Projects 27300616, 17300817, and 17300318) and National Natural Science Foundation of China via grant 11601457, Seed Funding Programme for Basic Research (HKU), and Basic Research Programme (JCYJ20180307151603959) of The Science, Technology and Innovation Commission of Shenzhen Municipality. Part of the work was done when J. Chen was visiting Department of Mathematics, University of Hong Kong. J. Chen would like to thank its hospitality.
%We would like to thank Professor Thomas Hou for stimulating discussions.
	
	%\section{Determine the random dimension of the solution space} \label{sec:StochasticDimension}
	
%	\section*{References}
	%	\bibliographystyle{plain}
	\bibliographystyle{siam}
	\bibliography{MDJpaper}
	
	%\section*{References}
	%\begin{thebibliography}{10}
	%\end{thebibliography}
	
\end{document}